\documentclass[11pt, a4paper]{article}
\usepackage{amsmath,amssymb,amsthm,amsfonts,latexsym,graphicx,subfigure}
\usepackage{color}
\usepackage[all,import]{xy}
\usepackage[numbers,sort&compress]{natbib}
\usepackage{indentfirst}
\usepackage{fancyhdr}
\usepackage{bbm}
\usepackage{accents,cases}
\usepackage{algorithm}
\usepackage{algorithmic}
\usepackage{multirow}
\usepackage{enumerate}
\usepackage[colorlinks,
linkcolor=red,
anchorcolor=black,
citecolor=blue
]{hyperref}
\usepackage{array}
\newcommand{\PreserveBackslash}[1]{\let\temp=\\#1\let\\=\temp}
\newcolumntype{C}[1]{>{\PreserveBackslash\centering}p{#1}}
\newcolumntype{R}[1]{>{\PreserveBackslash\raggedleft}p{#1}}
\newcolumntype{L}[1]{>{\PreserveBackslash\raggedright}p{#1}}

\DeclareMathOperator{\rank}{\ensuremath{rank}}

\makeatletter
\def\wbar{\accentset{{\cc@style\underline{\mskip8mu}}}}

\makeatother

\numberwithin{equation}{section}

\theoremstyle{plain}
\newtheorem{theorem}{Theorem}[section]
\newtheorem{defn} [theorem] {Definition}
\newtheorem{lemma} [theorem] {Lemma}
\newtheorem{remark}[theorem]{Remark}

\newtheorem{cor}[theorem]{Corollary}
\newtheorem{pro}[theorem]{Proposition}

\begin{document}
\bibliographystyle{unsrt}

\title{Algebraic  Surfaces with $p_g=q=1, K^2=4$ and  Genus 3   Albanese Fibration }
	\author{Songbo Ling}
	\date{}
	\maketitle

 \footnotetext{This work was supported by    China Scholarship Council ``High-level university graduate program'' (No.201506010011).}

	\begin{abstract}
		In this paper, we study the Gieseker moduli space  $\mathcal{M}_{1,1}^{4,3}$ of minimal surfaces with $p_g=q=1, K^2=4$ and   genus 3  Albanese fibration.  Under the assumption that  direct image of the canonical sheaf under the Albanese map    is decomposable,
		we find  two irreducible components of $\mathcal{M}_{1,1}^{4,3}$, one of  dimension 5 and the other of dimension 4.
		
			\end{abstract}
	
	\section{Introduction}
 Minimal surfaces of general type with $p_g=q=1$ have attracted the interest of  many authors (e.g. \cite{Cat81,CC91,CC93,CP06,Hor81,Pig09,Pol09,Rit1,Rit2}). For these surfaces,   one has $1\leq K^2\leq 9$.  By results of Mo\v{\i}\v{s}ezon \cite{Moi65},  Kodaira \cite{Kod68} and Bombieri \cite{Bom73},   these surfaces belong to a finite number  of families.

For such a surface $S$,  the Albanese map $f:S\rightarrow B:=Alb(S)$   is a fibration. Since the genus $g$ of a general    fibre of $f$ (cf. \cite{CP06} Remark 1.1) and $K_S^2$ are    differentiable invariants,  surfaces with different $g$  or  $K^2$  belong to  different connected components of the Gieseker moduli space. Hence one can study the moduli space of these surfaces  according to the pair ($K^2, g$).
	
	The case  $K^2=2$ has been accomplished    by Catanese \cite{Cat81} and Horikawa \cite{Hor81} independently: these surfaces have $g=2$ and  the moduli space is irreducible;  the case  $K^2=3$ has been studied completely by Catanese-Ciliberto \cite{CC91,CC93} and Catanese-Pignatelli \cite{CP06}: these surfaces have $g=2$ or $g=3$. Moreover,  there are  three irreducible connected components for surfaces with $g=2$ and one  irreducible connected component for  surfaces with $g=3$; for  the case  $K^2=4,g=2$,  Pignatelli \cite{Pig09} found eight disjoint irreducible  components of the moduli space under the assumption that the direct image of the bicanonical sheaf under the Albanese map is a direct sum of three line bundles.

For the case  $K^2=4, g=3$, there are some known examples (see e.g. \cite{Ish05,Pol09,Rit1,Rit2}), but the moduli space is still mysterious.
	Now denote  by $\iota$ the index  of the  paracanonical system (cf. \cite{CC91} for definition)   of $S$. By a result of Barja and Zucconi \cite{BZ00}, one has either  $\iota=g-1=2$ or $\iota=g=3$ (see Lemma \ref{BZ00 iota}). By \cite{CC91} Theorems 1.2 and 1.4, $\iota$ is a topological invariant, hence it is a deformation invariant and surfaces with different $\iota$ belong to different connected components of the moduli space.

    In  this paper we study the case $K^2=4,g=3, \iota=2$. Due to technical reasons, we  begin by assuming  that the general Albanese fibre is hyperelliptic. We call surfaces with these properties surfaces of type $I$ and  denote by $\mathcal{M}_I$ their image  in   $\mathcal{M}_{1,1}^{4,3}$.
    Our main result is the following

    \begin{theorem}
    	\label{theorem 4,3}
    	$\mathcal{M}_I$ consists of two disjoint irreducible subsets  $\mathcal{M}_{I_1}$ and $\mathcal{M}_{I_2}$ of dimension 4 and 3 respectively.  Moreover, $\mathcal{M}_{I_1}$ is contained in a 5-dimensional irreducible component of  $\mathcal{M}_{1,1}^{4,3}$ and $\mathcal{M}_{I_2}$ is contained in a 4-dimensional irreducible component of  $\mathcal{M}_{1,1}^{4,3}$.
    	For the general surface in these strata the general Albanese fibre is nonhyperelliptic.
    \end{theorem}

   \vspace{3ex}
   This paper  is organized as follows.

   In section 2,  we study the relative canonical map of $f$. We prove that  every Albanese fibre  of $S$ is 2-connected.
   The main ingredient for that  is Proposition \ref{2-connected},
   which gives a sufficient  condition for a    fibre of genus 3 to  be 2-connected.

   In section 3, we restrict to surfaces of type $I$, i.e. minimal surfaces with $p_g=q=1, K^2=4, g=3, \iota=2$ and hyperelliptic Albanese fibrations.  Using Murakami's structure theorem \cite{Mu12}, we divide   surfaces of type $I$ into two types according to the order of some torsion line bundle:   type $I_1$ and  type $I_2$ (cf. Definition \ref{classification}).
   Moreover,   we show that  the subspace $\mathcal{M}_{I_1}$ of $\mathcal{M}_{1,1}^{4,3}$ corresponding to  surfaces of type $I_1$  and the subspace $\mathcal{M}_{I_2}$ of $\mathcal{M}_{1,1}^{4,3}$ corresponding to surfaces  of type $I_2$  are two disjoint closed subset of $\mathcal{M}_{1,1}^{4,3}$.

   In section 4, we study surfaces of type $I_1$. We first construct  a family $M_1$ of surfaces of type $I_1$  using  bidouble covers of $B^{(2)}$, the  second symmetric product of an elliptic curve $B$. Then we show that  every surface of type $I_1$ is biholomorphic to some surface in $M_1$ and  that $\dim \mathcal{M}_{I_1}=4$.  After that we study the natural deformations of  the general surfaces of type $I_1$ and show that they give  a 5-dimensional irreducible subset $\mathcal{M}'_1$ of   $\mathcal{M}_{1,1}^{4,3}$.
   By computing $h^1(T_S)$ for a general surface $S\in M_1$, we prove that $\overline{\mathcal{M}'_1}$ is  an irreducible   component of $\mathcal{M}_{1,1}^{4,3}$.

   In section 5, we study surfaces of type $I_2$.  An interesting fact is that every surface of type $I_2$   also arises  from  a  bidouble cover of $B^{(2)}$, but  the  branch curve is   in a different linear equivalence class. Using a similar method to the one of  section 4, we show  that  $\dim\mathcal{M}_{I_2}=3$ and that $\mathcal{M}_{I_2}$ is   contained in  a 4-dimensional  irreducible   component of $\mathcal{M}_{1,1}^{4,3}$.

\vspace{2ex}
$\mathbf{Notation~and ~conventions.}$
Throughout this paper we work over the field $\mathbb{C}$ of complex numbers. We usually denote by $S$ a minimal  surface of general type with $p_g=q=1$  and by $S'$ the canonical model of $S$.

 We denote by $\Omega_S$  the sheaf of holomorphic 1-forms on $S$,   by    $T_S:=\mathcal{H}om_{\mathcal{O}_S}(\Omega_S,\mathcal{O}_S)$ the tangent sheaf of $S$ and       by $\omega_S:=\wedge^2\Omega_S$    the sheaf of holomorphic 2-forms on $S$. $K_S$ (or simply  $K$ if  no confusion)  is  the canonical divisor of $S$, i.e. $\omega_S\cong \mathcal{O}_S(K_S)$.   $p_g:=h^0(\omega_S), q:=h^0(\Omega_S)$.   The Albanese fibration of $S$ is denoted by  $f: S\rightarrow B:=Alb(S)$. We denote by $g$ the genus of a general fibre of $f$  and set $V_n:=f_*\omega_S^{\otimes n}$. For an elliptic curve $B$, we denote by $B^{(\iota)}$ the $\iota$-th symmetric product of $B$ and by $E_u(r,1)$ ($u$ is a point on $B$) the unique indecomposable rank $r$ vector bundle over $B$ with determinant $\mathcal{O}_B(u)$ (cf.  \cite{Ati57}).

 We denote by   $\mathcal{M}_{1,1}^{4,3}$    the Gieseker moduli space of   surfaces of general type  with $p_g=q=1,K^2=4$ and    genus 3 Albanese fibrations.
For divisors, we denote    linear equivalence by  `$\equiv$' and   algebraic equivalence  by `$\sim_{alg}$'.

	\section{The relative canonical map and 2-connectedness of Albanese fibres}
In this section, unless otherwise indicated, we always assume that $S$ is a minimal surface with $p_g=q=1,K^2=4$ and a genus 3 Albanese fibration $f:S\rightarrow B:=Alb(S)$. 

First we  recall the some definitions that we shall use later  from \cite{CC91} section 1. Let $t$ be a point on $B$ and set $K\oplus t:=K+f^*(t-0)$ (where $0$ is the neutral element of the elliptic curve $B$). Since $h^0(K)=p_g=1$  and $h^0(K\oplus t)=1+h^1(K\oplus t)$ (by Riemann-Roch), by the upper semicontinuity, there is a Zariski open subset $U\ni 0$ of $B$ such that for any $t\in U$,  $h^0(K\oplus t)=1$. We denote by  $K_t$   the unique effective divisor in $|K\oplus t|$ for any $t\in U$.

We define the {\em paracanonical incidence correspondence} to be the schematic closure  $Y$ (observe that it  is a divisor) in  $S\times B$ of the set $\{(x,t)\in S\times U|x\in K_t\}$.
Let $\pi_S: S\times B\rightarrow S$ and $\pi_B: S\times B\rightarrow B$ be the natural projections.
 We define $K_t$ as the fibre of $\pi_B|_Y: Y\rightarrow B$ over $t$ for any $t\in B\setminus U$. Note that  $Y$ provides a flat family of curves on $S$, which we denote by $\{K_t\}_{t\in B}$ (or simply $\{K\}$) and call it the {\em paracanonical system} of $S$.
The {\em index} $\iota$ of $\{K\}$ is the intersection number of $Y$ with the curve $\{x\}\times B$ for a general point $x\in S$, which is exactly the degree of the map $\pi_S|_Y: Y\rightarrow S$.

Now we  define a rational map   $w': S\dashrightarrow B^{(\iota)}$  as follows:  for a general point $x\in S$,  $w'(x):=(t_1,t_2,\cdots t_\iota)$ such that   $(\pi_S|_Y)^{-1}(x)=\{(x,t_1), (x, t_2), \cdots (x,t_\iota)\}$.
 We call $w'$ the
 {\em paracanonical map}   of $S$.

\vspace{3ex}
Since  $\deg V_1=1$ and  $\rank V_1=g$,   $V_1$ has a unique decomposition into indecomposable vector bundles $V_1=\bigoplus _{i=1}^kW_i$ with $\deg W_1=1$ and    $\deg W_i=0, H^0(W_i)=0$ $(2\leq i\leq k)$ (cf. \cite{CC91} p. 56). Let $w: S\dashrightarrow \mathbb{P}(V_1)$ be the relative canonical map of $f$. Then we have

\begin{lemma} [\cite{CC91} Theorem 2.3]
	\label{diagram}
	$\rank W_1=\iota$ and  $\rank W_i=1$ $(i=2,\cdots k)$. Moreover, $W_i$ $(i=2,\cdots k)$ are nontrivial torsion line bundles (see \cite{CP06} Remark 2.10) and we have the following  commutative diagram of rational maps
	$$\xymatrix{S\ar@{-->}[r]^w \ar@{-->}[dr]^{w'} & \mathbb{P}(V_1)\ar@{-->}[d]^\varphi \\ ~ & B^{(\iota)}=\mathbb{P}(W_1)
	}$$ where $\varphi$ is induced by the  natural inclusion: $W_1 \hookrightarrow V_1$.
\end{lemma}

\begin{lemma}
	\label{BZ00 iota}
	Let $S$ be a minimal surface with $p_g=q=1,K^2=4$ and a genus 3 Albanese fibration. Then we have either $\iota=2$ or $\iota=3$.
	\end{lemma}
\begin{proof}
	Since $K_{S/B}^2=K_S^2=4$ and $\Delta(f):=\chi(\mathcal{O}_S)-(g-1)(g(B)-1)=1$, we see that $\frac{K_{S/B}^2}{\Delta(f)}=4$. By \cite{BZ00} Theorem 2 and Lemma \ref{diagram}, 	we have either $\iota=g-1=2$ or $\iota=g=3$.
\end{proof}

Now we study the relative canonical map $w$ of $f:S\rightarrow B$.

\begin{lemma}
	\label{bpf}
	Let $F$ be a general fibre of $f$.  Then $|K_S+dF|$ is base point free for $d\gg 0$ and $w$ is a morphism.
\end{lemma}

\begin{proof}
	
	Denote by $|\mathfrak{m}|$ the movable part of $|K_S+dF|$ and by $\mathfrak{z}$ the fixed part of $|K_S+dF|$. Set  $S_0:=w(S)$. Denote by $T$ the (tautological) divisor on $\mathbb{P}(V_1)$ such that $\pi_*\mathcal{O}(T)=V_1$  and by $H$ the  fibre of $\pi: \mathbb{P}(V_1)\rightarrow B$.
	
	For $d>>0$, let $\xi: \mathbb{P}(V_1)\rightarrow \mathbb{P}^n$ be the holomorphic map defined by the linear system $|T+dH|$, where $n=h^0(T+dH)$. Let $\psi: S\dashrightarrow \mathbb{P}^n$ be the rational map defined by $|K_S+dF|$ (note that $h^0(T+dH)=h^0(V_1(d\cdot p))=h^0(K_S+dF)$, where $p=\pi(H)$). Then we have  $|w^*(T+dF)|\cong |\psi^*(K_S+dF)|$ and the following  diagram
	$$\xymatrix{ S\ar@{-->}[r]^w\ar@{-->}[rd]^\psi & \mathbb{P}(V_1)\ar[d]^\xi \\
		~ & \mathbb{P}^n  }$$
	commutes.	
	Hence  the indeterminacy points of  $w$ are exactly  the base points of  the movable part $|\mathfrak{m}|$ of $|K_S+dF|$. So we only need to show that $|K_S+dF|$ is base point free for $d\gg 0$.

	(i) If   $F$ is hyperelliptic,
	then the map $w: S\dashrightarrow S_0$ is of  degree 2.  Assuming  $S_0 \sim_{alg} 2T+\beta H$ for some $\beta$,   since $K_S, \mathfrak{m}, F$ are all nef divisors, we have
	$$4+8d=(K_S+dF)^2=\mathfrak{m}^2+\mathfrak{m}\mathfrak{z}+(K_S+dF)\mathfrak{z}\geq \mathfrak{m}^2 \geq 2\deg S_0 =2(T+dH)^2(2T+\beta H)=4+8d+2\beta.$$
	It follows that $\beta \leq 0$ and  that $\beta=0$ if and  only if  $K_S\mathfrak{z}=\mathfrak{m}\mathfrak{z}=F\mathfrak{z}=0$. Since  $K_S+dF$ is effective, big and  nef , by  \cite{ML89} Chap.I, Lemma 4.6, $K_S+dF$ is 1-connected. Since $K_S+dF=\mathfrak{m}+\mathfrak{z}$, we see that $\mathfrak{m}\mathfrak{z}=0$ if and only if $\mathfrak{z}=0$. Thus $\beta=0$ if and only if $|\mathfrak{m}|$ is base point free and $\mathfrak{z}=0$, i.e.
	$|K_S+dF|$ is base point free. Hence it suffices to show  that $\beta\geq 0$.
	
By Lemma \ref{BZ00 iota}, 	we have either $\iota=2$ or $\iota=3$. Now we discuss the two cases separately.

	If $\iota=2$, w.l.o.g. we can assume that  $V_1=E_{[0]}(2,1)\oplus N$ with $N$ a nontrivial torsion line bundle over $B$ (see Lemma \ref{diagram}).  Note that    $H^0(2T+\beta H)\cong  H^0(\pi_*\mathcal{O}_{\mathbb{P}(V_1)}(2T+\beta H))\cong H^0(S^2(V_1)(\beta\cdot p))$, where $p=\pi(H)$ is a point on $B$. Since $S^2(V_1)=\mathcal{O}_B(\eta_1)\oplus \mathcal{O}_B(\eta_2)\oplus\mathcal{O}_B(\eta_3)\oplus E_{[0]}(2,1)\otimes N\oplus N^{\otimes 2}$ (here $\eta_1,\eta_2,\eta_3$ are the three nontrivial 2-torsion points on $B$), we see that $h^0(2T+\beta H)>0$ only if $\beta\geq -1$.
	
	If $\beta =-1$, then $|2T-H|$ is nonempty if and only if $H=H_{\eta_i}:=\pi^*\mathcal{O}_B(\eta_i)$ for $i\in \{1,2,3\}$. Since  $h^0(2T-H_{\eta_i})=h^0(S^2(V_1)(-\eta_i))=1$, $|2T-H_{\eta_i}|$  contains a unique effective divisor $S_0$. Note that $S_0$ is a cone over a curve $C \sim _{alg} 2D-E$ lying on $B^{(2)}$, where $D$ (resp. $E$) is a section (resp. fibre)  of $B^{(2)}\rightarrow B$. Hence  $\varphi(S_0)=C$ is a curve.  By Lemma \ref{diagram}, we have $w'(S)=\varphi\circ w(S)=\varphi (S_0)=C$. On the other hand, one sees easily from the definition of the paracanonical map (cf. section 2.1) that $w'(S)=B^{(2)}$. Hence we get a contradiction.  Therefore  we have $\beta\geq 0$.
	
	If $\iota=3$, then  $V_1$ is indecomposable. Since  $S_0 \sim_{alg} 2T+\beta H$ and $S_0$ is effective, by \cite{CC93}  Theorem 1.13, we have $\beta\geq 0$.
	
	\vspace{2ex}
	(ii) If $F$ is nonhyperelliptic, then the map $w: S\dashrightarrow S_0$ is birational.  Assume that   $S_0\sim_{alg} \alpha T+\beta H$ for some $\alpha, \beta$.  Since  $F$ is of   genus 3, we have $T(\alpha T+\beta H)H=\alpha=4$.
	
	Since $w$ is birational, we have
	$$4+8d=(K_S+dF)^2\geq \mathfrak{m}^2\geq (T+dH)^2(\alpha T+\beta H)=\alpha+2d\alpha+\beta \geq 4+8d+\beta.$$
	Thus we have  $\alpha=4$ and  $\beta\leq 0$. Moreover $\beta=0$ if and only if  $|K_S+dF|$ is base point free. So it suffices to  show that   $\beta\geq 0$.
	
	Recall that we have either  $\iota=2$ or $\iota=3$.

	If $\iota=3$, then $V_1$ is indecomposable.  By \cite{CC93}  Theorem 1.13, $|4T+\beta H|\neq \emptyset$ if and only if $\beta\geq -1$. If $\beta=-1$, by \cite{CC93} Theorem 3.2, a general element $S_t$ in $|4T-H|$ is a smooth surface with ample canonical divisor. Note that  $S\rightarrow S_0$ is  the minimal resolution of $S_0$.  Since $S_0$ is irreducible,  $K_{S_0}$ is Cartier, and  $K_{S_0}^2=K_{S_t}^2=3$  ($K_{S_0}\sim^{alg} T|_{S_0}$, so $K_{S_0}^2=T^2(4T-H)=3$), by  \cite{KSB} Proposition 2.26, we have $4=K_S^2\leq K_{S_0}^2=3$, a contradiction. Therefore we have $\beta\geq 0$.
	
If $\iota=2$,  we consider the paracanonical system of $S$.
By Lemma \ref{diagram}, we have the following commutative diagram:
	$$\xymatrix{S\ar@{-->}[r]^w \ar@{-->}[dr]^{w'} & S_0\subset \mathbb{P}(V_1)\ar@{-->}[d]^\varphi \\ ~ & B^{(2)}=\mathbb{P}(W_1)
	}$$
	Since $S_0\sim_{alg}4T+\beta H$ and   $\varphi$ is the projection  induced by the  natural inclusion: $W_1 \hookrightarrow V_1$, we see that the degree of the map $\varphi|_{S_0}: S_0\dashrightarrow B^{(2)}$ is 4 (fibrewise, it  maps a quartic  plan  curve to a line). Hence $\deg w'=4$. Now write $\{K\}=\{M\}+Z$, where $\{M\}$ is the movable part of $\{K\}$ and $Z$ is the fixed part of $\{K\}$.  Note that  the paracanonical map is  defined by $\{M\}$.  
	
	Let $\lambda: \tilde{S}\rightarrow S$ be the shortest composition of blow-ups such that the movable part $\{\tilde{M}\}$ of $\lambda^*\{M\}$ is base point free.  Then we have $4=K_S^2=M^2+MZ+K_SZ\geq M^2\geq \tilde{M}^2=\deg (\lambda\circ\varphi|_{S_0})\cdot 1=4$. In particular, we have  $MZ=0$ and $M^2=\tilde{M}^2$. By the 2-connectedness of canonical divisors and the definition of $\lambda$, we see that $Z=0$ and $\{M\}$ is base point free, i.e. $\{K\}$ is base point free. 
	
Now we consider the linear system $|K_S+F_{t_0}| (t_0\in B)$.  For any fixed point $x\in S$, since $\{K_t\}_{t\in B}$ and $\{F_t\}_{t\in B}$ are both base point free, if we take general $t\in B$,   then the  divisor $K_t+F_{t_0-t}\in |K_S+F_{t_0}|$ does not contain $x$. Hence $|K_S+F_{t_0}|$ is base point free.  So    $|K_S+dF|$ is base point free for any $d\geq 1$. Therefore  the relative canonical map is base point free. 
	
\end{proof}

Since the restriction map $H^0(S, K_S+dF)\rightarrow H^0(F, K_F)$ is surjective  for $d\gg 0$ (cf. Horikawa \cite{Hor77} Lemmas 1 and  2), we   get the following

\begin{cor}
	\label{KF bpf}
	$|K_F|$ is base point free for any fibre $F$ of $f$.
\end{cor}

Catanese-Franciosi (\cite{CF93} Corollary 2.5) proved that:  if $C$ is  a 2-connected curve of  genus $p_a(C)\geq 1$ lying on a smooth algebraic surface, then $|K_C|$ is base point free. However,  the converse is not true in general, e.g. if we take $C$  the union of two distinct smooth   fibres of a genus 2 fibration, then  $|K_C|$ is base point free, but $C$ is not even 1-connected. Now we  show that the converse is  true in the following  case:

\begin{pro}
	\label{2-connected}
	Let  $f:S\rightarrow B$ be a relatively minimal   genus 3  fibration and let $F$ be any fibre of $f$. If  $|K_F|$ is base point free, then $F$ is 2-connected.
\end{pro}

To prove Proposition \ref{2-connected}, we need the following four lemmas.
\begin{lemma}[Zariski's Lemma, \cite{BHPV} Chap. III, Lemma 8.2]
	\label{Zariski' lemma}
	Let $F=\sum n_iC_i$ ($n_i>0$, $C_i$ irreducible) be a fibre of the fibration $f:S\rightarrow B$. Then we have
	
	(i) $C_iF=0$ for all $i$;
	
	(ii) If $D=\sum_i m_iC_i$, then $D^2\leq 0$, and $D^2=0$ holds  if and only if $D=rF$ for some $r\in \mathbb{Q}$.
\end{lemma}

\begin{lemma}[\cite{CF93} Corollary 2.5]
	\label{base point}
	Let $C$ be a curve of genus $p_a(C)\geq 1$ lying on a smooth algebraic surface. If $C$ is  1-connected, then the base points of $|K_C|$ are precisely the points $x$ such that there exists a decomposition $C=Y+Z$ with $YZ=1$, where $x$ is smooth for $Y$ and $\mathcal{O}_Y(x)\cong \mathcal{O}_Y(Z)$.
	
\end{lemma}

\begin{lemma} [\cite{ML89}  Chap. I, Lemmas 2.2 and 2.3]
	\label{elliptic cycle}
	Assume that  $D$ is  a 1-connected  divisor  on a smooth algebraic surface  and  let $D_1\subset D$ be minimal subject to the condition $D_1(D-D_1)=1$. Then $D_1$ is 2-connected and either
	
	(i) $D_1\subset D-D_1$  or
	
	(ii) $D_1$ and $D-D_1$ have no common components.
\end{lemma}

\begin{lemma}[\cite{ML89} Chap. I, Proposition 7.2]
	\label{vanishing}
	Let $D$ be a 2-connected divisor  with $p_a(D)=1$ on a smooth algebraic surface, and let $\mathcal{L}$ be an invertible  sheaf on $D$ such that $\deg\mathcal{L}|_C\geq 0$ for each component $C$ of $D$. If $\deg\mathcal{L}|_D=1$, then $\mathcal{L}\cong \mathcal{O}_D(x)$ with $x$ a smooth  point of $D$ and $H^0(\mathcal{L})$ is generated by one section vanishing only at $x$.
\end{lemma}

Now we prove Proposition \ref{2-connected}.
\begin{proof}[Proof of proposition \ref{2-connected}]
	
	If $F$ is not 2-connected, then either (i) $F$ is not 1-connected, or  (ii) $F$ is 1-connected, but not 2-connected. We discuss the two cases separately.

	(i) If $F$ is not 1-connected, then $F$ must be a multiple fibre, i.e. $F=mF'$ with $F'$ 1-connected.  Since $K_SF=mK_SF'=4$ and $K_SF'$ is even,  we see   $m=2$. Thus we have $K_SF'=2$, $F'^2=0$  and $p_a(F')=2$. Since $\mathcal{O}_{F'}(F')$ is a nontrivial 2-torsion line bundle on $F'$, by \cite{BHPV} Chapter II Lemma 12.2, $h^1(\omega_{F'}(F'))= h^0(\mathcal{O}_{F'}(-F'))=0$. Since $\chi(\omega_{F'}(F'))=\chi(\mathcal{O}_S(K_S+F))-\chi(\mathcal{O}_S(K_S+F'))=\frac{K_SF'}{2}=1$, we know $h^0(\omega_F|_{F'})=h^0(\omega_{F'}(F'))=1$. Hence    $|\omega_F|_{F'}|$ has a base point  and so does $|K_F|$, a contradiction.
	
	(ii)  Assume that $F$ is 1-connected, but not 2-connected.
	Let $D\subset F$ realize a minimum of $K_SD$ among the subdivisors such that   $D(F-D)=1$. Let  $E:=F-D$. By Zariski's Lemma, we have  $D^2=E^2=-1$. By Lemma \ref{elliptic cycle}, $D$ is 2-connected and either
	
	(1) $D\subset E$   or
	
	(2) $D$ and $E$ have no common components.

	In case (2), since $DE=1$, $D$ intersects $E$ transversely in  one point $x$, which must be a smooth point of both curves.  Note that  $\mathcal{O}_D(x)\cong \mathcal{O}_D(E)$.  By Lemma \ref{base point},  $x$ is a base point of $|K_F|$, a contradiction.

	We study now case (1), i.e. $D\subset E$. Since $D^2=D(F-E)=-1$ and $K_S(D+E)=4$, we have  $K_SD=1$ and $K_SE=3$. In particular, we have  $p_a(D)=1$.  If $D$ is irreducible, we can always find a smooth point $x$ on $D$ such that $\mathcal{O}_D(x)\cong \mathcal{O}_D(E)$ (cf. \cite{Hart77} Chap. IV, Ex. 1.9). By Lemma \ref{base point},   $x$ is a base point of $|K_F|$, a contradiction.

	If $D$ is reducible, since $K_S$ is nef and $K_SD=1$, there is a unique irreducible component $C_0$ of $D$ such that $K_SC_0=1$.
	Write $D-C_0=\sum_{i\geq 1} m_iC_i$ with $C_i$  distinct irreducible curves, then we have $K_SC_i=0$ for $i\geq 1$. Hence  $C_i$ $(i\geq 1)$ are $(-2)$-curves.
	Since $D$ is 2-connected,  $DC_i=(D-C_i)C_i+C_i^2\geq 0$. Since $-1=D^2=C_0D+\sum_{i\geq 1}m_iC_iD$, we have $-1\geq C_0D=C_0^2+C_0(D-C_0)\geq C_0^2+2$, thus $C_0^2\leq -3$. Since $C_0$ is irreducible and $K_SC_0=1$, we  get  $C_0^2=-3$ and $C_0$ is a  smooth rational curve. Thus we get  $C_0D=-1, C_iD=0$ $(i\geq 1)$, and consequently   $C_0E=1, C_iE=0$ $(i\geq 1)$.
	
	Now let $\mathcal{L}:= \mathcal{O}_D(E)$,  so that  $\deg\mathcal{L}|_C\geq 0$ for any component $C$ of $D$ and $\deg\mathcal{L}|_D=1$. By Lemma \ref{vanishing},  we have $\mathcal{L}\cong \mathcal{O}_D(x)$ with $x$ a smooth point of $D$.  Hence   $x$ is a base point of $|K_F|$ by Lemma \ref{base point}, a contradiction.
	
	Therefore  $F$ is 2-connected.
\end{proof}

\begin{remark}	
	The key point in the above proof for case (ii) is that we can find a 2-connected elliptic cycle (i.e. $K_SD=1, D^2=-1$) $D\subset F$   such that  $D(F-D)=1$ and $\mathcal{L}:=\mathcal{O}_D(F-D)$ satisfies the condition of Lemma 2.6. Using  a similar argument, one can  get an analogous result for genus 2 fibrations, i.e.
	
	Let  $F$ be any  fibre of a relatively  minimal genus 2  fibration  $f:S\rightarrow B$. If $|K_F|$ is base point free, then $F$ is 2-connected.
\end{remark}

Combining  Corollary \ref{KF bpf} and  Proposition \ref{2-connected}, we get the following
\begin{theorem}
	\label{2-connected theorem}
	Let $S$ be a minimal surface with $p_g=q=1, K^2=4$ and a genus 3  Albanese fibration.  Then every Albanese fibre of $S$ is 2-connected.
\end{theorem}

	\section{Murakami's structure theorem for genus 3  hyperelliptic fibrations}
	In this section, we use Murakami's   structure theorem for genus 3  hyperelliptic fibrations to study  surfaces of type $I$. First we recall briefly the notation and  idea in Murakami's structure theorem (see \cite{Mu12} for details).
	
\subsection{Murakami's structure theorem}	
We first introduce the admissible 5-tuple $(B, V_1, V_2^+, \sigma, \delta)$ in Murakami's structure theorem and then explain the structure  theorem.

The 5-tuple $(B, V_1, V_2^+, \sigma, \delta)$ is defined as follows:

$B$: any  smooth   curve;

$V_1$:  any  locally free sheaf  of rank 3 over $B$ ;

$V_2^+$: any   locally free sheaf  of rank   5 over $B$;

 $\sigma$ :   any  surjective morphism $S^2(V_1)\rightarrow V_2^+$;

$\delta$: any  morphism  $(V_2^-)^{\otimes 2}\rightarrow \mathcal{A}_4$. Here $V_2^-$ and $\mathcal{A}_4$ are defined as follows:
letting  $L:=\ker\sigma$,  which gives an exact sequence
$$0\rightarrow L\rightarrow S^2(V_1)\stackrel{\sigma}\rightarrow V_2^+\rightarrow 0.$$
We set  $V_2^-:=(\det V_1)\otimes L^{-1}$  and  define  $\mathcal{A}_n$  as  the cokernel of the injective morphism $L\otimes S^{n-2}(V_1)\rightarrow S^n(V_1)$ induced by  the inclusion $L\rightarrow S^2(V_1)$.

\vspace{3ex}
Set now  $\mathcal{A}:=\bigoplus _{n=0}^\infty \mathcal{A}_n$  and let $\mathcal{S}(V_1)$ be the symmetric $\mathcal{O}_B$-algebra of $V_1$.   Via the natural  surjection $\mathcal{S}(V_1)\rightarrow \mathcal{A}$, the algebra structure of $\mathcal{S}(V_1)$ induces a  graded $\mathcal{O}_B$-algebra structure on $\mathcal{A}$.  Let $\mathcal{C}:=\mathbf{Proj}(\mathcal{A})$,
$\mathcal{R}:=\mathcal{A}\oplus (\mathcal{A}[-2]\otimes V_2^-)$ and $X:=\mathbf{Proj}(\mathcal{R})$.

The 5-tuple $(B, V_1, V_2^+, \sigma, \delta)$ is said to be admissible if:

(i) $\mathcal{C}$ has at most RDP's as singularities;

(ii) $X$ has at most  RDP's as singularities.

  \begin{theorem}[Murakami's structure theorem, cf. \cite{Mu12} Theorem 1]
 The isomorphism classes of relatively  minimal  genus 3 hyperelliptic fibrations with all fibres 2-connected are in one to one correspondence with the isomorphism classes of admissible 5-tuples
 $(B, V_1, V_2^+, \sigma, \delta)$.
\end{theorem}
More precisely (cf. \cite{Mu12} Propositions 1 and  2), given a relatively  minimal genus 3 hyperelliptic  fibration $f:S\rightarrow B$  with all fibres 2-connected and setting  $V_n:=f_*\omega_{S/B}^{\otimes n}$,  we can define its associated    5-tuple
$(B, V_1, V_2^+, \sigma, \delta)$ as follows:

$B$  is the base curve;

$V_1=f_*\omega_{S/B}$;

$V_2^-$: the hyperelliptic fibration $f$ induces an involution of $S$, which acts on $V_2=f_*\omega_{S/B}^{\otimes 2}$. We define  $V_2^+$ and $V_2^-$  to be  the natural decomposition of $V_2$ into eigensheaves $V_2=V_2^+\oplus V_2^-$ with  eigenvalues $+1$ and $-1$ respectively;

$\sigma: S^2V_1\rightarrow V_2^+$ is the natural morphism induced by the multiplication structure of the relative canonical algebra $\mathcal{R}=\bigoplus_{n=1}^\infty V_n$ of $f$;

$\delta: (V_2^-)^{\otimes 2}\rightarrow V_4^+$ is the natural morphism induced by the multiplication of $\mathcal{R}$.

Moreover, the  associated 5-tuple is admissible.

\vspace{3ex}
Conversely, given an admissible 5-tuple $(B, V_1, V_2^+, \sigma, \delta)$, we have the  graded $\mathcal{O}_B$-algebras $\mathcal{S}(V_1), \mathcal{A}$, $\mathcal{R}$  and varieties $\mathcal{C}, X$. Note that $\mathcal{C}\in |\mathcal{O}_{\mathbb{P}(V_1)}(2)\otimes \pi^*L^{-1}|$ is a conic bundle determined by $\sigma$, and $X$ is the  double cover of  $\mathcal{C}$ with branch divisor determined by $\delta\in Hom_{\mathcal{O}_B}((V_2^-)^{\otimes 2}, \mathcal{A}_4)\cong H^0(\mathcal{C}, \mathcal{O}_{\mathcal{C}}(4)\otimes (\pi|_{\mathcal{C}})^*(V^-_2)^{-2})$, where $\pi: \mathbb{P}(V_1)=\mathbf{Proj}(\mathcal{S}(V_1))\rightarrow B$ is the natural projection.
Let  $\bar{f}: X\rightarrow B$ be the natural projection and  $v: S\rightarrow X$ be the minimal resolution of $X$. Then   $f:=v\circ \bar{f}: S\rightarrow B$ is  a relatively  minimal genus 3 hyperelliptic  fibration with all fibres 2-connected.	Moreover, we have $f_*\omega_{S/B}=V_1$ and   $S$ has  the following numerical invariants:
$$\chi(\mathcal{O}_S)=\deg V_1+2(b-1),$$
$$K_S^2=4\deg V_1-2\deg L+16(b-1),$$
where  $b$ is the genus of $B$.

\subsection{Application to surfaces of type $I$}
	Let $S$ be a  surface of type $I$, i.e. a minimal surface with $p_g=q=1,K^2=4$ and hyperelliptic Albanese fibration $f:S\rightarrow B$ such that $g=3, \iota=2$ (note that  ``$V_1$ decomposable" means $\iota=2$  by Lemma \ref{BZ00 iota}). By Theorem \ref{2-connected theorem},  every fibre of $f$ is 2-connected.  Hence we can use Murakami's structure theorem to study $f$.
	
	For later convenience we fix a group structure on $B$,  denote by $0$ its neutral element and by $\eta_1,\eta_2,\eta_3$ the three  nontrivial 2-torsion points.
	
	By Lemma \ref{diagram},  we can assume  $V_1=E_{[0]}(2,1)\oplus N$,
	where   $N$ is  a nontrivial torsion line bundle over $B$.   Now we  use Murakami's structure theorem to study the  order of $N$.
	In the notation we introduced in section 2.4, we have:

	\begin{lemma}
		\label{LN}
		$L\cong N^{\otimes 2}.$
	\end{lemma}
	\begin{proof}  Since $\det V_1=N(0)$,  we have $V_2^-=(\det V_1)\otimes L^{-1}=N(0)\otimes L^{-1}.$  From section 2.4, we have $\rank L=\rank  S^2(V_1)-\rank V_2^+=1$ and $\deg L=\frac{1}{2}(4\deg V_1+16(b-1)-K_S^2)=0$, i.e.
		$L$ is a line bundle of degree $0$. Hence $V_2^-$ is a line bundle of degree $1$.
		
		Tensoring the   exact sequence
		$$0\rightarrow L\rightarrow S^2(V_1) \rightarrow V_2^+\rightarrow 0$$ with $N^{- 2}$, we get the associated cohomology long exact sequence
		$$H^1(L\otimes N^{-2}) \rightarrow H^1(S^2(V_1)\otimes N^{- 2}) \rightarrow H^1(V_2^+\otimes N^{- 2})\rightarrow 0.$$
		
		Since $h^0(V_2\otimes N^{-2})=h^0(\omega_S^{\otimes 2}\otimes f^*N^{-2})=5$ and $h^0(V_2^-\otimes N^{-2})=1$ (as $\deg (V_2^-\otimes N^{-2})$=1), we get
		$h^0(V_2^+\otimes N^{-2})=4$.  By Riemann-Roch for vector bundles over a smooth curve (cf. \cite{BHPV} Chap. II, Theorem 3.1),  we have  $h^1(V_2^+\otimes N^{-2})=h^0(V_2^+\otimes N^{-2})-\deg (V_2^+\otimes N^{-2})=0$ (note that $\deg (V_2^+\otimes N^{-2})=\deg (V_2^+)=\deg (V_2)-\deg (V_2^-)=4$). Since $h^1(S^2(V_1)\otimes N^{-2})\geq 1$,   we get $h^0(L\otimes N^{-2})\geq1$. Since  $\deg(L\otimes N^{-2})=0$, we deduce that $L\cong N^{\otimes 2}$ (cf. \cite{Ati57} Theorem 5).	
	\end{proof}
	
	\begin{lemma}	
		\label{split}
		The exact sequence $$0\rightarrow L\rightarrow S^2(V_1)=(\bigoplus _{i=1}^3\mathcal{O}_B(\eta_i))\oplus E_{[0]}(2,1)\otimes N \oplus L\rightarrow V_2^+\rightarrow 0$$
		splits.
	\end{lemma}
	\begin{proof} From  the proof of Lemma \ref{LN}, we know that $L\rightarrow S^2(V_1)$ induces an isomorphism $H^1(L\otimes N^{-2})\cong H^1(S^2(V_1)\otimes N^{-2})(\neq 0)$.  Thus  the composition  map
		$$0\rightarrow L\rightarrow S^2(V_1)\rightarrow L $$ is nonzero (here the last map is the natural projection), hence  it is an isomorphism. Therefore the above  exact sequence splits.
	\end{proof}

	\begin{remark}
		\label{divisor restricion}
		As in \cite{CP06} Lemma  6.14 or \cite{Pig09} section 1.2, since the map $ L\otimes S^2(V_1)\rightarrow S^4(V_1)$ factors as
		$$L\otimes S^2(V_1)\rightarrow S^2(V_1) \otimes S^2(V_1)\rightarrow  S^4(V_1),$$
		Lemma \ref{split} implies that  the exact sequence $$0\rightarrow L\otimes S^2(V_1)\rightarrow S^4(V_1) \rightarrow \mathcal{A}_4 \rightarrow 0$$ also splits  (see \cite{Pig09} section 1.2 p.5 for details). Hence  the branch curve $\delta\in |\mathcal{O}_{\mathcal{C}}(4)\otimes (\pi|_{\mathcal{C}})^* (\det V_1\otimes L^{-1})^{-2}|$ comes from an effective divisor in $ |\mathcal{O}_{\mathbb{P}(V_1)}(4)\otimes \pi^*(\det V_1\otimes L^{-1})^{-2}|$.
	\end{remark}
	
	By  Lemmas \ref{LN},  \ref{split} and Remark \ref{divisor restricion},   we have
	$$\det V_1=\mathcal{O}_B(0)\otimes N, $$
	$$ S^2(V_1)=(\bigoplus_{i=1}^3\mathcal{O}_B(\eta_i))\oplus E_0(2,1)\otimes N\oplus L,$$
	$$ V_2^-=\det V_1\otimes L^{-1}=\mathcal{O}_B(0)\otimes N^{-1}.$$

	\begin{lemma}
		\label{order of torsion}
		$L^{\otimes 2}\cong N^{\otimes 4} \cong \mathcal{O}_B$.
	\end{lemma}
	\begin{proof}
		Recall that $V_1=E_{[0]}(2,1)\oplus N$, where   $N$ is  a nontrivial torsion line bundle over $B$.
		By Atiyah (cf. \cite{Ati57}), we  have  $$S^4(V_1)=S^4(E_{[0]}(2,1))\oplus (S^3(E_{[0]}(2,1))\otimes N) \oplus (S^2(E_{[0]}(2,1)\otimes N^{\otimes 2}) \oplus  (E_{[0]}(2,1)\otimes N^{\otimes 3}) \oplus N^{\otimes 4},$$
		$$S^4(E_{[0]}(2,1))=\mathcal{O}_B(2\cdot 0)\oplus \mathcal{O}_B(2\cdot 0)\oplus\mathcal{O}_B(\eta_1+\eta_2)\oplus \mathcal{O}_B(\eta_2+\eta_3)\oplus \mathcal{O}_B(\eta_3+\eta_1),$$
		$$S^3(E_{[0]}(2,1))=E_{[0]}(2,1)(0)\oplus E_{[0]}(2,1)(0)$$
		$$S^2(E_{[0]}(2,1))=\mathcal{O}_B(\eta_1)\oplus \mathcal{O}_B(\eta_2)\oplus\mathcal{O}_B(\eta_3).$$
		
		By \cite{Ati57} Lemma 15, assuming that  $\mathcal{E}$ is an indecomposable  vector bundle of rank $r$ and degree $d$ over an elliptic curve $B$, then $h^0(\mathcal{E})=d$ if $d>0$; $h^0(\mathcal{E})=0$ or $1$ if $d=0$. Moreover by \cite{Ati57} Theorem 5, if $d=0$, then $h^0(\mathcal{E})=1$ if and only if $\det \mathcal{E}\cong \mathcal{O}_B$.
		
		Hence we have
		$$H^0(S^4(V_1) \otimes \mathcal{O}_B(-2\cdot 0)\otimes L)=H^0(S^4(E_{[0]}(2,1))(-2\cdot 0)\otimes L)=H^0((\mathcal{O}_B\oplus \mathcal{O}_B \oplus N_1\oplus N_2\oplus N_3 )\otimes L),
		$$
		where $N_i:= \mathcal{O}_B(\eta_i-0)$ $(i=1,2,3)$.
		
		If $L^{\otimes 2}\ncong \mathcal{O}_B$, then $H^0(S^4(V_1) \otimes \mathcal{O}_B(-2\cdot 0)\otimes L)=0$ and $|\mathcal{O}_{\mathbb{P}(V_1)}(4)\otimes \pi^*(V_2^-)^{-2}|=\emptyset$,   a contradiction.  Hence we have $L^{\otimes 2}\cong \mathcal{O}_B$ and the  result  follows from Lemma \ref{LN}.
		
	\end{proof}
	
	Using Remark \ref{divisor restricion} and Lemma \ref{order of torsion}, we can decide  now the linear systems of the conic bundle $\mathcal{C}$ and the branch divisor $\delta$ in Murakami's structure theorem (cf. section 2.4):
	
	$ \mathcal{C}\in|\mathcal{O}_{\mathbb{P}(V_1)}(2)\otimes \pi^*L^{-1}|\cong |\mathcal{O}_{\mathbb{P}(V_1)}(2)|\cong  \mathbb{P}(H^0(S^2(V_1))), $
	$\delta\in |\mathcal{O}_{\mathbb{P}(V_1)}(4)\otimes \pi^*(V_2^-)^{-2}||_{\mathcal{C}},$ where  $|\mathcal{O}_{\mathbb{P}(V_1)}(4)\otimes \pi^*(V_2^-)^{-2}|\cong |\mathcal{O}_{\mathbb{P}(V_1)}(4)\otimes \pi^*\mathcal{O}_B(-2\cdot 0)|\cong  \mathbb{P}(H^0(S^4(V_1)\otimes \mathcal{O}_B(-2\cdot 0)))$.

	\vspace{3ex}
	Now we  divide surfaces of type $I$ into two types according to the order of $N$.
	\begin{defn}
		\label{classification}
		Let $S$ be a surface of type $I$ and assume $V_1=E_{[0]}(2,1)\oplus N$ with $N$ a nontrivial torsion line bundle (cf. Lemma \ref{diagram}).  We  call  $S$ of type $I_1$ if $N^{\otimes 2}\cong \mathcal{O}_B$;  we call $S$ of type $I_2$ if $N^{\otimes 2}\ncong \mathcal{O}_B$ and $N^{\otimes 4}\cong \mathcal{O}_B$.
		
		Denote by $\mathcal{M}_{I_1}$ the  the subspace  of $\mathcal{M}_{1,1}^{4,3}$ corresponding to  surfaces of type $I_1$,   and by $\mathcal{M}_{I_2}$ the subspace of $\mathcal{M}_{1,1}^{4,3}$ corresponding to surfaces  of type $I_2$.
		Then we have  $\mathcal{M}_I=\mathcal{M}_{I_1}\cup \mathcal{M}_{I_2}$.
	\end{defn}

	\begin{pro}
		\label{disjoint}
		$\mathcal{M}_{I_1}$ and 	$\mathcal{M}_{I_2}$ are two disjoint     Zariski closed subsets of $\mathcal{M}^{4,3}_{1,1}$.
	\end{pro}
	\begin{proof}
		Since  $N$ is a torsion line bundle of order 2 for surfaces of type $I_1$,  and   it  is a torsion line bundle of order 4 for surfaces of type $I_2$, we have  $\overline{\mathcal{M}_{I_1}}\cap \overline{\mathcal{M}_{I_2}}=\emptyset$. Now we show that $\mathcal{M}_{I_1}$ is a  Zariski closed subset of $\mathcal{M}^4_{1,1}$.
		By a similar argument, one can show that $\mathcal{M}_{I_2}$ is also a Zariski closed subset of $\mathcal{M}^4_{1,1}$.

		By \cite{Cat13} Theorem 24,   given  two minimal  surfaces of general type $S_1,S_2$ with their respective canonical models $S_1',S_2'$, then
		$S_1$ and $S_2$ are deformation equivalent $\Leftrightarrow$ $S_1'$ and $S_2'$ are deformation equivalent.
		Hence  it suffices to show: if $p:\mathcal{S}\rightarrow T$ is    a smooth connected 1-parameter   family of minimal surfaces such that for $0\neq t\in T$,  $\mathcal{S}_t:=p^{-1}(t)$ is a surface of type $I_1$, then   $\mathcal{S}_0=p^{-1}(0)$ is also a surface of type $I_1$.
		
		For  $0\neq t\in T$, a general  Albanese fibre of $\mathcal{S}_t$ is hyperelliptic of genus 3  and $V_1$ is decomposable. Since the genus of the Albanese fibration and the number of the direct summands of $V_1$ are deformation invariants,  we see that  a  general Albanese  fibre of $\mathcal{S}_0$  is also of genus 3 and $V_1$ of $\mathcal{S}_0$  is also decomposable.
		Moreover,  since  a general Albanese fibre of $\mathcal{S}_t$ is  hyperelliptic,
		a  general  Albanese fibre of $\mathcal{S}_0$ is also hyperelliptic.  Otherwise we  would get a flat  family of irreducible smooth curves $ C\rightarrow T$, whose   central fibre  is a nonhyperelliptic curve  and  whose general fibre  is a hyperelliptic curve,  a contradiction.
		
		Hence $\mathcal{S}_0$ is also  a surface of type $I$.  Since   $\mathcal{M}_I=\mathcal{M}_{I_1}\cup \mathcal{M}_{I_2}$ and  $\overline{\mathcal{M}_{I_1}}\cap \overline{\mathcal{M}_{I_2}}=\emptyset$, we conclude that $\mathcal{S}_0$ is a surface of type $I_1$.  Therefore $\mathcal{M}_{I_1}$ is a Zariski closed subset of $\mathcal{M}^{4,3}_{1,1}$.
	\end{proof}

\section{Surfaces of type $I_1$}
In this section, we focus on surfaces of type $I_1$. First we show that  surfaces of type $I_1$ are in one to one correspondence with  some  bidouble covers of  $B^{(2)}$.

\subsection{Bidouble covers of $B^{(2)}$}

Recall that (cf. \cite{Cat84} Proposition 2.3)  a  smooth bidouble cover $h: S\rightarrow X$ is uniquely determined by the data of {\em effective divisors}  (sometimes we also call them {\em branch divisors}) $D_1,D_2,D_3$ and divisors $L_1,L_2,L_3$ such that $D=D_1\cup D_2\cup D_3$ has normal crossings and
\begin{equation}
\quad 2L_i\equiv D_j+	D_k ,
\quad D_k+L_k\equiv L_i+L_j.
\quad   \{i,j,k\}=\{1,2,3\}
\end{equation}
As Manetti \cite{Man94} pointed out,   these facts are true in a more general situation where  $X$ is smooth and $S$ is normal (in this case, each $D_i$ is still reduced, but $D$ may have other singularities except for ordinary double points).

\vspace{3ex}
Let $p: B^{(2)}=\{(x,y)| x\in B, y\in B, (x,y)\sim (y,x)\}\rightarrow B$  be the natural projection defined by $(x,y)\mapsto x+y$. Set $D_u:=\{(u,x)| x\in B\}$  a section of $p$  and  $E_u:=\{(x,u-x)|x\in B\}$ a fibre of $p$. Now  we construct a family of surfaces of type $I_1$ using bidouble covers of $B^{(2)}$.
\begin{pro}
	\label{construction 1}
	Let $h: S'\rightarrow X:=B^{(2)}$ be a  bidouble cover determined by   effective  divisors $D_1\equiv2D_0,D_2\equiv 4D_0-2E_0,D_3=0$, and divisors $L_1\equiv 2D_0-E_{\eta_i}, L_2\equiv D_0, L_3\equiv 3D_0-E_{\eta_i}$  such that $S'$ has at most RDP's as singularities.  Then the minimal resolution $\nu: S\rightarrow S'$ of $S'$
	yields  a surface $S$  of type $I_1$.
\end{pro}
\begin{proof}
	Let $G=(\mathbb{Z}/2\mathbb{Z})^2=\{1,\sigma_1,\sigma_2,\sigma_3\}$ be the Galois group of the bidouble cover $h$ and let    $R_i$ be the divisorial part of $Fix(\sigma_i)$.  Set  $R:=R_1\cup R_2\cup R_3$. Then we have $D_i=h(R_i)$ and $K_{S'}=h^*K_X+R$.
	
	Since $D:=D_1\cup D_2\cup D_3\equiv6D_0-2E_0$, $K_X\equiv-2D_0+E_0$ and $\chi(\mathcal{O}_X)=0$,  by \cite{Cat84}  (2.21) and (2.22),    we have
	$$K_{S'}^2=(2K_X+D)^2=4,
	\chi(\mathcal{O}_{S'})=4\chi(\mathcal{O}_X)+\frac{1}{2}K_XD+\frac{1}{8}(D^2+\Sigma_i D_i^2)=1.$$
	Moreover,  for $i=1,2$,  one has
	$$h^i(\mathcal{O}_{S'})=h^i(h_*\mathcal{O}_{S'})=h^i(\mathcal{O}_X)+h^i(\mathcal{O}_X(-L_1))+h^i(\mathcal{O}_X(-L_2)+h^i(\mathcal{O}_X(-L_3))=1.$$
	Since $S'$ has at most RDP's as singularities and  $K_{S'}$ is ample,  we see that $S$ is minimal, $K_S^2=K_{S'}^2=4$,  $p_g(S)=h^2(\mathcal{O}_S)=h^2({\mathcal{O}_{S'}})=1$ and $q(S)=h^1(\mathcal{O}_S)=h^1({\mathcal{O}_{S'}})=1$.
	
	The bidouble  cover $h: S'\rightarrow X$  can be decomposed into two  double covers $h_1: Y\rightarrow X$ with ${h_1}_*\mathcal{O}_Y=\mathcal{O}_X\oplus \mathcal{O}_X(-L_2)$, and $h_2: S'\rightarrow Y$ with branch curve $h_1^*D_2$. Note that the general  fibre of $Y\rightarrow B$ is an irreducible smooth rational curve,  which   intersects   $h_1^*D_2$  at 8 points.  Hence  the general  fibre of $f':= p\circ h : S'\rightarrow B$ (and also the  general fibre of  $f:=f'\circ \nu: S\rightarrow B$) is irreducible and  hyperelliptic  of genus 3. By the universal property of the Albanese map  and  of the Stein factorization, we see that $B=Alb(S)$  and  $f$ is  the Albanese fibration of $S$. Therefore $S$ has a genus 3 hyperelliptic Albanese fibration.
	
	Since $V_1=f_*\omega_S=f'_*\omega_{S'}$, we have
	$$h^0(V_1 \otimes \mathcal{O}_B(0-\eta_i))=h^0(\omega_{S'}\otimes f'^*\mathcal{O}_B(0-\eta_i))=2.$$
	Since $\deg(V_1)=1$, by \cite{Ati57} Lemma 15,  $V_1$ must be decomposable. By \cite{BZ00} Theorem 2 and  Lemma \ref{diagram},   we know  that $\iota=2$ and $V_1=E_{[0]}(2,1)\oplus N$ with $N$ a nontrivial torsion line bundle over $B$. Again by \cite{Ati57} Lemma 15, we get $N\cong \mathcal{O}_B(\eta_i-0)$.   Therefore $S$ is a surface of type $I_1$.
\end{proof}

Denote by  $M_1$ the family of minimal  surfaces $S$ obtained as  the minimal resolution of  a bidouble cover  $h: S'\rightarrow X=B^{(2)}$ as in Proposition \ref{construction 1}, and by $\mathcal{M}_1$ the image of $M_1$ in $\mathcal{M}_{I_1}$. Then we have

\begin{lemma}
	\label{dim M1}
	$\dim\mathcal{M}_1=4$.
\end{lemma}
\begin{proof}
	The moduli space of  $B^{(2)}$ has dimension 1. Since we have fixed the neutral element $0$ for  $B$, only a finite subgroup of $Aut(B^{(2)})$ acts on our data, and quotienting by it does not affect the dimension.    Since  $h^0(D_1)=h^0(2D_0)=h^0(S^2E_{[0]}(2,1))=3$ (cf. Lemma \ref{order of torsion}) and  $h^0(D_2)=h^0(-2K_X)=2$ (see \cite{Cat81} Proposition 10),  we have
	$$\dim\mathcal{M}_1=1+\dim |D_1|+\dim |D_2|=1+2+1=4.$$
\end{proof}

Next we show that the converse of Proposition \ref{construction 1} is also true.

($*$) For  the remainder  of  this section, we always assume that $S$ is a   surface of type $I_1$  and that  $S'$ is the  canonical model of $S$.

\begin{pro}
	\label{converse 1}
	$S'$ is a bidouble cover of  $B^{(2)}$ determined by effective   divisors $D_1\equiv2D_0,D_2\equiv 4D_0-2E_0,D_3=0$, and divisors $L_1\equiv 2D_0-E_{\eta_i}, L_2\equiv D_0, L_3\equiv 3D_0-E_{\eta_i}$.
\end{pro}

Since the proof is long, we divide it into  three steps in the following three lemmas:

(1)  (Lemma \ref{finite})  there is a finite morphism $h: S'\rightarrow B^{(2)}$ of degree 4;

(2)  (Lemma \ref{bidouble cover}) the morphism $h$ is a bidouble cover with branch divisors $(D_1,D_2,D_3)$ as  stated above;

(3)  (Lemma \ref{divisor class})  up to an automorphism of $B^{(2)}$, $L_1,L_2,L_3$ satisfy  the above  linear equivalence relations.

\vspace{2ex}

To prove (1), we first  introduce the map $h$. Since the relative canonical map  $w: S\rightarrow \mathbb{P}(V_1)$ factors as the composition  $\nu: S\rightarrow S'$ (the map contracting $(-2)$-curves) and $\mu: S'\rightarrow \mathbb{P}(V_1)$. Let $h:=\varphi\circ\mu: S'\dashrightarrow B^{(2)}$. By Lemma \ref{diagram}, we have the following commutative diagram
$$\xymatrix{S'\ar[r]^\mu\ar@{-->}[dr]^h & \mathbb{P}(V_1)\ar@{-->}[d]^\varphi \\ ~ & B^{(2)}
}$$ where  $w=\mu\circ\nu$ and  $w'=h\circ\nu$. Since $w$ is a morphism (cf. Lemma \ref{bpf}) and the general Albanese fibre of $S$ is hyperelliptic, we see that   $\mu: S'\rightarrow \mathcal{C}:=\mu(S')\subset \mathbb{P}(V_1)$ is a finite double cover. Moreover  $\mathcal{C}$ is exactly the conic bundle in Murakami's  structure theorem.  Now we prove (1).

\begin{lemma}
	\label{finite}
	The map  $h: S'\dashrightarrow B^{(2)}$ is  a finite morphism of degree 4.
\end{lemma}

\begin{proof}
	Since  $\mu: S'\rightarrow \mathcal{C}$ is a finite double cover,   it suffices to show that $\varphi|_{\mathcal{C}}: \mathcal{C}\dashrightarrow B^{(2)}$ is also a finite double cover.
	To prove this, we need  to study the equation of $\mathcal{C}\subset\mathbb{P}(V_1)$ and use the definition of   $\varphi$.
	
	To get global relative coordinates on  fibres of $\mathbb{P}(V_1)$, first  we  take a unramified double cover of $B$.
	Since $N$ is a 2-torsion line bundle, we can find  a unramified double cover  $\phi: \tilde{B}\rightarrow B $  such that $\phi^*N\cong \mathcal{O}_{\tilde{B}}$ and  $\phi^*0=\tilde{0}+\eta$ for some nontrivial  2-torsion point $\eta\in \tilde{B}$, where $\tilde{0}$ is the neutral element in the group structure of $\tilde{B}$,  and such that $\phi(\tilde{0})=0$. Moreover, by \cite{Ish05} Theorem 2.2 and Lemma 2.3, we have
	$\phi^*E_{[0]}(2,1) \cong \mathcal{O}_{\tilde{B}}(x)\oplus \mathcal{O}_{\tilde{B}}(x')$, where $x, x'$ are two points on $\tilde{B}$ such that  $\mathcal{O}_B(\phi_*(x)-0)\cong N$ (cf. \cite{Fri98} Chap. 2, Proposition 27) and $x'=x\oplus \eta$ in the group law of $\tilde{B}$.
	
	Set $\tilde{E}:=\phi^*(E_{[0]}(2,1)\oplus N)$ and $\tilde{\mathcal{C}}:=\Phi^*\mathcal{C}$. Then we have the following commutative diagram:
	$$\xymatrix
	{\mathcal{\tilde{C}}\subset \mathbb{P}(\tilde{E})\ar[rrr]^{\Phi}\ar[d]^{\tilde{\varphi}}&&&\mathcal{C}\subset \mathbb{P}(V_1)\ar[d]^{\varphi} \\
		\tilde{X}:=	\mathbb{P}( \mathcal{O}_{\tilde{B}}(x)\oplus \mathcal{O}_{\tilde{B}}(x')) \ar[rrr] \ar[d]^{\tilde{p}} &&& X= \mathbb{P}(E_{[0]}(2,1)) \ar[d]^p \\
		\tilde{B}\ar[rrr]^{\phi} &&&B}$$
	where $\tilde{\varphi}: \mathbb{P}(\tilde{E})\rightarrow \mathbb{P}( \mathcal{O}_{\tilde{B}}(x)\oplus \mathcal{O}_{\tilde{B}}(x'))$ is the natural projection induced by the injection $\mathcal{O}_{\tilde{B}}(x)\oplus \mathcal{O}_{\tilde{B}}(x')\rightarrow \tilde{E}=\mathcal{O}_{\tilde{B}}(x)\oplus \mathcal{O}_{\tilde{B}}(x')\oplus \phi^*N$.
	
	Note that the unramified double cover $\Phi: \mathbb{P}(\tilde{E})\rightarrow \mathbb{P}(V_1)$ induces an involution $T_\eta$ on $\mathbb{P}(\tilde{E})$. Let $J:=\{1,T_\eta\}$ be the group generated by $T_\eta$. Then for any divisor $D$ on $\mathbb{P}(V_1)$, we have $H^0(\mathbb{P}(V_1), D)\cong H^0(\mathbb{P}(\tilde{E})(\Phi^*D))^J$ (the $J$-invariant part of $H^0(\mathbb{P}(\tilde{E})(\Phi^*D))$).
	
	From  the commutative diagram above,  to show that   $\varphi|_{\mathcal{C}}$ is a finite double cover,  it suffices  to  show   that $\tilde{\varphi}|_{\tilde{\mathcal{C}}}: \tilde{\mathcal{C}}\rightarrow \mathbb{P}(\mathcal{O}_{\tilde{B}}(\tilde{0})\oplus \mathcal{O}_{\tilde{B}}(\eta))$  is a finite double cover.
	
	Take global relative coordinates $y_1:\mathcal{O}_{\tilde{B}}(x) \rightarrow \tilde{E}$, $y_2:\mathcal{O}_{\tilde{B}}(x') \rightarrow \tilde{E}$, $y_3:\mathcal{O}_{\tilde{B}} \rightarrow \tilde{E}$ on fibres of $\mathbb{P}(\tilde{E})$. In notation of section 4.2, we have  $\mathcal{C}\in|\mathcal{O}_{\mathbb{P}(V_1)}(2)|$, hence $\mathcal{\tilde{C}}$ is a $J$-invariant divisor in $|\mathcal{O}_{\mathbb{P}(\tilde{E})}(2)|$.
	Therefore the  equation of $\tilde{\mathcal{C}}\subset \mathbb{P}(\tilde{E})$ can be written as
	\begin{equation}
	\label{equation of C}
	f_1=a_1y_1^2+a_2y_2^2+a_3y^2_3+a_4y_1y_2+a_5y_1y_3-a_6y_2y_3,
	\end{equation}
	where $a_1,a_2\in H^0(\mathcal{O}_{\tilde{B}}(2x)),  a_3\in H^0(\mathcal{O}_{\tilde{B}}), a_4\in H^0(\mathcal{O}_{\tilde{B}}(x+x'), a_5\in H^0(\mathcal{O}_{\tilde{B}}(x))$ and  $a_6\in H^0(\mathcal{O}_{\tilde{B}}(x'))$. Since the action of $T_\eta^*$ is $y_1\mapsto y_2, y_2\mapsto y_1,y_3\mapsto -y_3$ and $\tilde{\mathcal{C}}$ is $J$-invariant, we see that  $T_\eta^*a_1=a_2$ and $ T_\eta^*a_5=a_6$.
	
	Since  finite double cover is a local property, we can check this locally. 	 Choose a local coordinate $t$ for the base curve $B$. Then $(t,(y_1:y_2:y_3))$ is a local coordinate on  $\mathbb{P}(\tilde{E})$ and  $(t,(y_1:y_2))$ is a local coordinate on $\tilde{X}$. 	The action of  $\tilde{\varphi}$ is locally like $(t,(y_1:y_2:y_3)) \mapsto (t,(y_1:y_2))$. From the equation of $\mathcal{\tilde{C}}$,   to  show   that $\tilde{\varphi}|_{\tilde{\mathcal{C}}}$ is a finite double cover,   it suffices to show that $a_3\neq 0$.

	\vspace{2ex}
	If $a_3=0$, then  $C:=\{y_1=y_2=0\} \subset \tilde{\mathcal{C}}$. Recall that
	the branch divisor $\delta$ of $u: S'\rightarrow \mathcal{C}$ is contained in $|\mathcal{O}_{\mathcal{C}}(4)\otimes \pi^*\mathcal{O}_B(-2\cdot 0)||_{\mathcal{C}}\cong |\mathcal{O}_{\mathbb{P}(V_1)}(4)\otimes \pi^*\mathcal{O}_B(-2\cdot 0)||_{\mathcal{C}}$ and it is reduced. Hence $\tilde{\delta}:=\Phi^*\delta$ is a $J$-invariant reduced divisor in $|\mathcal{O}_{\mathbb{P}(\tilde{E})}(4)\otimes \tilde{\pi}^*\mathcal{O}_{\tilde{B}}(-2\cdot \tilde{0}-2\eta)||_{\mathcal{\tilde{C}}}$. Since $2x\equiv 2x' \equiv 2\eta \equiv 2\cdot \tilde{0}$ and $x+x'\equiv \tilde{0}+\eta \not\equiv 2\cdot \tilde{0}$,  one sees easily that $y_1^4,y_1^2y_2^2, y_2^4$ is a basis of $H^0(\mathcal{O}_{\mathbb{P}(\tilde{E})}(4)\otimes \tilde{\pi}^*\mathcal{O}_{\tilde{B}}(-2\cdot \tilde{0}-2\eta))$. Let $f_2$ be the equation of $\tilde{\delta}$ on $\tilde{\mathcal{C}}$. Then $f_2$ has the form $f_2=b_1y_1^4+b_2y_1^2y_2^2+b_3y_2^4$, where $b_1,b_2,b_3\in\mathbb{C}.$
	
	Note that  $C$ is contained in  $\tilde{\delta}$.  Since  the Jacobian matrix of $(f_1,f_2)$ at any point of $C$ has the form
	\begin{gather*}
	\begin{pmatrix}
	\frac{\partial f}{\partial t}& a_5y_3 & a_6y_3 & 0 \\
	0 & 0 & 0 & 0
	\end{pmatrix}
	\end{gather*}
	which has rank 1,   $\tilde{\delta}$ is singular  along $C$. Hence $\tilde{\delta}$ is nonreduced,    a contradiction.
	
	Therefore  $\varphi: \mathcal{C}\rightarrow B^{(2)}$ is a finite double cover and $h$ is a finite morphism of degree 4.
\end{proof}

\begin{remark}
	From that above Lemma, one sees easily that fibrewise, the composition map $S'\rightarrow \mathcal{C}\rightarrow B^{(2)}$ is just: a genus 3 hyperelliptic curve $\stackrel{2:1}\rightarrow$ a conic curve  in $\mathbb{P}^2$ $\stackrel{2:1}\rightarrow$ $\mathbb{P}^1$.
\end{remark}

Now we prove (2).
\begin{lemma}
	\label{bidouble cover}
	The morphism 	$h: S'\rightarrow B^{(2)}$   is a bidouble cover  with branch divisors    $D_1\equiv2D_0,D_2\equiv 4D_0-2E_0,D_3=0$.
\end{lemma}
\begin{proof}
	As in section 4.1, we denote by $H$ the fibre of $\pi: \mathbb{P}(V_1)\rightarrow B$ and by $T$ the divisor on $\mathbb{P}(V_1)$ such that $\pi_*\mathcal{O}(T)=V_1$. Similarly,  we denote by $\tilde{H}$ the fibre of $\tilde{\pi}: \mathbb{P}(\tilde{E})\rightarrow \tilde{B}$ and by $\tilde{T}$ the divisor  on $\mathbb{P}(\tilde{E})$ such that $\tilde{\pi}_*\mathcal{O}(\tilde{T})=\tilde{E}$.
	By  Lemma \ref{finite}, the	ramification divisor of $\tilde{\varphi}|_{\tilde{\mathcal{C}}}$ on $\tilde{\mathcal{C}}$ is defined by
	$$(a_5y_1+a_6y_2)^2-4a_3(a_1y_1^2+a_2y_2^2+a_4y_1y_2)=f_1=0$$
	and is  linearly equivalent to $2\tilde{T}|_{\tilde{\mathcal{C}}}$. Thus  the ramification divisor of $\varphi|_{\mathcal{C}}$ on $\mathcal{C}$ is linearly equivalent to $2T|_{\mathcal{C}}$  (which is the $J$-invariant part of $2\tilde{T}|_{\tilde{\mathcal{C}}}$).   From the definition of $\varphi$, we know that   $D_0=\varphi(T)$.  Hence  the branch divisor of $\varphi|_{\mathcal{C}}$ is linearly equivalent to    $2D_0$.

	Since  $h^0((4T-2H_0)|_{\mathcal{C}})=h^0(\varphi^*(4D_0-2E_0)|_{\mathcal{C}})=h^0(4D_0-2E_0)+h^0(3D_0-2E_0)$ (double cover formula) $=h^0(4D_0-2E_0)$  (cf. \cite{CC93} Theorem 1.13), we get $|(4T-2H_0)|_{\mathcal{C}}|=(\varphi|_{\mathcal{C}})^*|4D_0-2E_0|$.  Hence  the branch divisor of $\mu:S'\rightarrow \mathcal{C}$ is invariant under the involution $\sigma_1'$ of $\mathcal{C}$ induced by the double cover $\varphi|_{\mathcal{C}}:\mathcal{C} \rightarrow B^{(2)}$. So $\sigma_1'$ lifts to an involution $\sigma_1$ on $S'$. Note that the double cover $\mu:S'\rightarrow \mathcal{C}$ induces another involution $\sigma_2$ on $S'$. Hence we get a group
	$G:=\{1,\sigma_1,\sigma_2,\sigma_3:=\sigma_1\circ \sigma_2 \}$ acting effectively on $S'$, and the quotient $S'/G$ is nothing but $B^{(2)}$.
	
	Therefore   $h: S'\rightarrow B^{(2)}$ is a bidouble cover. Moreover, the three branch divisors of $h$  are $D_1=h(Fix(\sigma_1))\equiv 2D_0, D_2=h(Fix(\sigma_2))\equiv 4D_0-2E_0 , D_3=h(Fix(\sigma_3))=0$.
\end{proof}

Now we prove (3).
\begin{lemma}
	\label{divisor class}
	Up to an automorphism of $B^{(2)}$, we can assume the data  $(L_1,L_2,L_3)$ of $h:S'\rightarrow X:=B^{(2)}$ to be $L_1\equiv 2D_0-E_{\eta_i}, L_2\equiv D_0, L_3\equiv 3D_0-E_{\eta_i}$.
\end{lemma}
\begin{proof}
	Since $$h^1(\mathcal{O}_S)=h^1(\mathcal{O}_{S'})=h^1(h_*\mathcal{O}_{S'})=h^1(\mathcal{O}_X)+h^1(\mathcal{O}_X(-L_1))+h^0(\mathcal{O}_X(-L_2))+h^0(\mathcal{O}_X(-L_3))=1$$
	and  $h^1(\mathcal{O}_X)=1$, we see  $h^1(\mathcal{O}_X(-L_1))=0$. In particular, we have  $L_1\not\equiv -K_X$. Since $2L_1\equiv D_2+D_3\equiv 4D_0-2E_0$, we have  $L_1\equiv 2D_0-E_{\eta_i}$ for a nontrivial 2-torsion point $\eta_i \in B$.
	Since $L_2+L_3\equiv D_1+L_1\equiv  4D_0-E_{\eta_i}$, there are three   choices for  $(L_2,L_3)$:
	
	(i) $L_2\equiv D_0, L_3\equiv 3D_0-E_{\eta_i}$;
	
	(ii) $L_2\equiv D_{\eta_i}, L_3\equiv 3D_0-E_0$;
	
	(iii) $L_2\equiv D_{\eta_j}(j\neq i), L_3\equiv 3D_0-E_{\eta_k}$.
	
	Now we show that  for fixed $(D_1,D_2,D_3,L_1)$  above, the three choices (i) (ii) (iii) for $(L_2, L_3)$ are equivalent up to an automorphism of $X=B^{(2)}$. The automorphism $(x,y)\mapsto (x+\eta_i,y+\eta_i)$ on $X$ fixes fibres of $X\rightarrow B$ and translates $D_u$ to $D_{u+\eta_i}$. Hence  it fixes $(D_1,D_2,D_3,L_1)$ and maps $(L_2,L_3)$ in  (i) to $(L_2,L_3)$  in  (ii).  Similarly, the automorphism $(x,y)\mapsto (x+\eta_j,y+\eta_j)$ fixes $(D_1,D_2,D_3,L_1)$ and maps $(L_2,L_3)$ in  (i) to  $(L_2,L_3)$ in  (iii).

	Therefore, up to an automorphism of $B^{(2)}$, we can   assume the data $(L_1,L_2,L_3)$ of  $h$ to be 	$L_1\equiv 2D_0-E_{\eta_i}, L_2\equiv D_0, L_3\equiv 3D_0-E_{\eta_i}$.
\end{proof}

Combining   Propositions \ref{construction 1} and  \ref{converse 1} together, we get the following
\begin{theorem}
	\label{type 1}
	If $h: S'\rightarrow B^{(2)}$ is a  bidouble cover determined by   branch  divisors $D_1\equiv2D_0,D_2\equiv 4D_0-2E_0,D_3=0$, and divisors $L_1\equiv 2D_0-E_{\eta_i}, L_2\equiv D_0, L_3\equiv 3D_0-E_{\eta_i}$  such that $S'$ has at most RDP's as singularities, then the minimal resolution  $S$ of $S'$  is a surface of type $I_1$.  Conversely, if $S$ is  a surface of type $I_1$, then  the canonical model  $S'$ of $S$  is a bidouble cover of  $B^{(2)}$  (where $B=Alb(S)$) determined by the branch  divisors $(D_1,D_2,D_3)$ and divisors $(L_1,L_2,L_3)$ in the respective linear equivalence classes above.
\end{theorem}

The following corollary follows easily from Lemma \ref{dim M1} and  Theorem \ref{type 1}.
\begin{cor}
	\label{dimension 1}
	$\mathcal{M}_1=\mathcal{M}_{I_1}$. In particular, we have $\dim\mathcal{M}_{I_1}=4$.
\end{cor}

\subsection{Natural deformations of smooth bidouble covers}
Let $S$ be a general surface of type $I_1$. Then we have a smooth bidouble cover $h: S\rightarrow X=B^{(2)}$  determined by branch  divisors $D_1\equiv2D_0,D_2\equiv 4D_0-2E_0,D_3=0$, and divisors $L_1\equiv 2D_0-E_{\eta_i}, L_2\equiv D_0, L_3\equiv 3D_0-E_{\eta_i}$. Since $D_3=0$,    $h$ is a simple bidouble cover (see \cite{Cat85} Definition 22.4).
In this subsection, we study the natural deformations of $S$. (see \cite{Cat85} p. 75 for details)

Let $L_1':=D_0, L_2':=2D_0-E_{\eta_i}$ and let $z_1,z_2$ be the fibre coordinates relative to the two summands of $V:=\oplus _{j=1}^2\mathcal{O}_X(-L_j')$. Let $x_i$ be a section of $\mathcal{O}_X(D_i)$ with div$(x_i)=D_i$ ($i=1,2$).   By \cite{Cat85} p. 75,    $S$ is a subvariety of $V$ defined by equations:

\begin{equation}
\label{equation  simple bidouble cover}
\quad z^2_1=x_1,
\quad  z_2^2=x_2.
\end{equation}
and a natural deformation $Y$ of $S$  is defined by equations
\begin{equation}
\label{equation natural deformation simple bidouble cover}
\quad z^2_1=x_1 +b_1z_2,
\quad  z_2^2=x_2+b_2z_1.
\end{equation}
with $b_1\in H^0(\mathcal{O}_X(D_1-L_2'))$, $b_2\in H^0(\mathcal{O}_X(D_2-L_1'))$.

Note that $h^0(\mathcal{O}_X(D_1-L_2'))=h^0(E_{\eta_i})=h^0(\mathcal{O}_B(\eta_i))=1$.  By \cite{CC93} Theorem 1.13, we have $H^0(\mathcal{O}_X(D_2-L_1'))=H^0(3D_0-2E_0)=0$, hence we always have $b_2=0$.

\vspace{3ex}
Denote by $M'_1$ the family of all surfaces arising as natural deformations of some general surface of type $I_1$, and by $\mathcal{M}'_1$ the image of $M'_1$ in $\mathcal{M}_{1,1}^{4,3}$. Let $\overline{\mathcal{M}'_1}$ be the Zariski closure of $M'_1$ in $\mathcal{M}_{1,1}^{4,3}$. Then we have
\begin{pro}
	\label{dim M'1}
	$\dim \mathcal{M}'_1=5$ and  $\mathcal{M}_{I_1}$ is a 4-dimensional subspace of $\overline{\mathcal{M}'_1}$.
\end{pro}
\begin{proof}
	Since there is one  parameter for $X=B^{(2)}$ (see Lemma \ref{dim M1}). From equations (\ref{equation natural deformation simple bidouble cover}), we see that $\dim \mathcal{M}'_1=1+\dim |D_1|+ \dim |D_2|+ h^0(\mathcal{O}_X(D_1-L_2'))=1+2+1+1=5$.
\end{proof}

\begin{remark}
	\label{remark general surface in M'1}
	From equations (\ref{equation natural deformation simple bidouble cover}), It is easy to see that a natural deformation $Y$ of $S$ is a bidouble cover of $X$ if and only if $b_1=0$ (since we always have $b_2=0$). By Theorem \ref{type 1}, $Y$ has a genus 3 hyperelliptic Albanese fibration if and only if $b_1=0$. Since $\dim \mathcal{M}_{I_1}<\dim \mathcal{M}'_1$, we see that  a general surface in $M'_1$ has a genus  3 nonhyperelliptic Albanese fibation.
\end{remark}

\subsection{$h^1(T_S)$ for a general surface $S$ of type $I_1$}
In this section  we calculate $h^1(T_S)$ for a general surface $S$ of type $I_1$.  Note that for general choices of $D_1\in  |2D_0|$ and $D_2\in |4D_0-2E_0|$, $D_1$, $D_2$ are both irreducible  smooth  curves and they intersect transversally. Hence     $S$ is a smooth bidouble cover of  $X:=B^{(2)}$ determined by effective divisors $(D_1,D_2,D_3)$ and divisors $(L_1,L_2,L_3)$ as in Theorem \ref{type 1}.

By Riemann-Roch, we have  $h^0(T_S)-h^1(T_S)+h^2(T_S)=2K_S^2-10\chi(\mathcal{O}_S)=-2$. Since    $h^0(T_S)=0$, we have  $h^1(T_S)=h^2(T_S)+2=h^0(\Omega_S\otimes \omega_S)+2$.
By  \cite{Cat84} Theorem 2.16, we have
$$H^0(\Omega_S\otimes \omega_S)\cong H^0(h_*(\Omega_S\otimes \omega_S))$$
$$=H^0(\Omega_X(logD_1, logD_2, logD_3)\otimes\omega_X)\oplus(\bigoplus_{i=1}^3 H^0(\Omega_X(logD_i)\otimes \omega_X(L_i))).$$

Hence to calculate $h^1(T_S)$, it suffices to calculate $h^0(\Omega_X(logD_1, logD_2, logD_3)$ and $h^0(\Omega_X(logD_i)\otimes \omega_X(L_i))$ $(i=1,2,3)$.

\begin{lemma}
	\label{cotangentsheaf}
	$\Omega_X=\mathcal{O}_X\oplus \omega_X.$
\end{lemma}
\begin{proof}  Since $p: X=B^{(2)}\rightarrow B$ is a $\mathbb{P}^1$-bundle, by \cite{Hart77} Chap. III, Ex. 8.4, we have the following exact sequence
	$$ 0\rightarrow \Omega_{X/B}\rightarrow (p^*E_{[0]}(2,1))(-1)\rightarrow \mathcal{O}_X\rightarrow 0.$$
	Since $$\wedge^2( (p^*E_{[0]}(2,1))(-1))= \mathcal{O}_X(-2)\otimes p^*\mathcal{O}_B(0)\cong \Omega_{X/B}\otimes \mathcal{O}_X,$$
	we see  $\Omega_{X/B}\cong \mathcal{O}_X(-2)\otimes p^*\mathcal{O}_B(0)\cong \omega_X$.
	
	On the other hand, we have the exact sequence
	$$0\rightarrow p^*\omega_B\rightarrow \Omega_X\rightarrow \Omega_{X/B}\rightarrow 0$$
	i.e. 	$$0\rightarrow \mathcal{O}_X\rightarrow \Omega_X\rightarrow \omega_X\rightarrow 0.$$
	Since $Ext^1(\omega_X, \mathcal{O}_X)\cong H^1(\omega_X^{-1})=H^1(\mathcal{O}_X(2D_0-E_0))$ and $h^1(\mathcal{O}_X(2D_0-E_0))=h^1(S^2(E_{[0]}(2,1))(-0))=h^0(S^2(E_{[0]}(2,1))(-0))=0$  (cf. proof of Lemma \ref{order of torsion}),
	we see that $\Omega_X=\mathcal{O}_X\oplus \omega_X.$
\end{proof}

\begin{lemma}
	\label{term 0}
	$h^0(\Omega_X(logD_1, logD_2, logD_3)\otimes\omega_X)=1.$
\end{lemma}

\begin{proof} Let $g: Y\rightarrow X$ be the   smooth   double cover with $g_*\mathcal{O}_Y=\mathcal{O}_X\oplus \mathcal{O}_X(-\epsilon)$, where $\epsilon\equiv -K_X$.    Then  $Y=B\times B$ (see \cite{Cat81} Proposition 4) and we have  $\Omega_Y\cong \mathcal{O}_Y\oplus \mathcal{O}_Y$.  Since (see \cite{Cat84} Proposition 3.1)
	$$H^0(\Omega_Y\otimes \omega_Y)\cong H^0(\Omega_X(logD_2)\otimes \omega_X)\oplus H^0(\Omega_X\otimes \omega_X(\epsilon))$$
	$$=H^0(\Omega_X(logD_2)\otimes \omega_X)\oplus H^0(\Omega_X),$$
	$h^0(\Omega_Y\otimes \omega_Y)= h^0(\mathcal{O}_Y\oplus \mathcal{O}_Y)=2$ and $h^0(\Omega_X)=1$, we have   $h^0(\Omega_X(logD_2)\otimes \omega_X)=1$. Since $\Omega_X(logD_2)\otimes \omega_X \subset \Omega_X(logD_1, logD_2, logD_3)\otimes\omega_X$, we see  $h^0(\Omega_X(logD_1, logD_2, logD_3)\otimes\omega_X)\geq 1$.

	On the other hand, by \cite{Cat84} (2.12), we have the following exact sequence
	\begin{equation}
	0\rightarrow \Omega_X\otimes \omega_X\rightarrow \Omega_X(logD_1, logD_2, logD_3)\otimes\omega_X\rightarrow \bigoplus_{i=1}^3\mathcal{O}_{D_i}(K_X)\rightarrow 0.
	\end{equation}
	Note that
	$h^0(\Omega_X\otimes \omega_X)=0$.
	Since $K_X D_1=-2$,  we have  $h^0(\mathcal{O}_{D_1}(K_X))=0$. Since  $D_2$ is an irreducible elliptic curve  in the rational pencil $|-2K_X|$ (cf. \cite{Cat81} Proposition 6), we know that  $D_2|_{D_2}\equiv 0$ and  $(K_X+D_2)|_{D_2}\equiv 0$, thus  $K_X|_{D_2}\equiv 0$ and $h^0(\mathcal{O}_{D_2}(K_X))=1$.  Hence we have  $h^0(\Omega_X(logD_1, logD_2, logD_3)\otimes\omega_X)\leq 1$.
	
	Therefore we get $h^0(\Omega_X(logD_1, logD_2, logD_3)\otimes\omega_X)=1$.
\end{proof}

\begin{lemma}
	\label{term 1}
	$h^0(\Omega_X(logD_1)\otimes \omega_X(L_1))=0$.
\end{lemma}
\begin{proof} Consider the following exact sequences
	$$0\rightarrow \mathcal{O}_X(-2D_0+E_0-E_{\eta_i}) \rightarrow \mathcal{O}_X(E_0-E_{\eta_i})\rightarrow \mathcal{O}_{D_1}(E_0-E_{\eta_i})\rightarrow 0$$
	$$0\rightarrow \Omega_X(E_0-E_{\eta_i})\rightarrow \Omega_X(logD_1)\otimes \omega_X(L_1) \rightarrow \mathcal{O}_{D_1}(E_0-E_{\eta_i})\rightarrow 0$$
	Since  $h^0(\mathcal{O}_X(-2D_0+E_0-E_{\eta_i}))=h^0(p_*\mathcal{O}_X(-2D_0)(0-\eta_i))=0$ (cf. \cite{BHPV} Chap. I, Theorem 5.1),
	$h^1(\mathcal{O}_X(-2D_0+E_0-E_{\eta_i}))=h^1(\mathcal{O}_X(E_{\eta_i}))=h^1(\mathcal{O}_B(\eta_i))=0$,
	$h^0(\mathcal{O}_X(E_0-E_{\eta_i}))=h^0(\mathcal{O}_B(0-\eta_i))=0$, we have
	$h^0(\mathcal{O}_{D_1}(E_0-E_{\eta_i}))=0$.
	Since moreover $h^0( \Omega_X(E_0-E_{\eta_i}))=h^0(\mathcal{O}_X(E_0-E_{\eta_i}))+h^0(\mathcal{O}_X(-2D_0+E_{\eta_i}))=0$,  we get $h^0(\Omega_X(logD_1)\otimes \omega_X(L_1))=0$.
\end{proof}

\begin{lemma}
	\label{term 2}
	$h^0(\Omega_X(logD_2)\otimes \omega_X(L_2))=1$.
\end{lemma}

\begin{proof}Let $g: Y\rightarrow X$ be the smooth   double cover  in Lemma \ref{term 0}.   Since
	$$H^0(\Omega_Y\otimes \omega_Y(g^*L_2))\cong H^0(\Omega_X(log(D_2)\otimes \omega_X(L_2)) \oplus H^0(\Omega_X\otimes\omega_X (\epsilon+L_2)),$$
	$h^0(\Omega_X\otimes \omega_X(\epsilon+L_2))=h^0(\Omega_X(D_0))=h^0( \mathcal{O}_X(D_0))+h^0( \mathcal{O}_X(-D_0+E_0))$,  $h^0(\mathcal{O}_X(D_0))=h^0(E_{0}(2,1))=1$, $h^0( \mathcal{O}_X(-D_0+E_0))=0$ (cf. \cite{BHPV} Chap. I, Theorem 5.1) and
	$h^0(\Omega_Y\otimes \omega_Y(g^*L_2))=h^0(g_*(\mathcal{O}_Y\oplus \mathcal{O}_Y)\otimes \mathcal{O}_X(L_2))=2 h^0(\mathcal{O}_X(L_2))+2 h^0(\mathcal{O}_X(-D_0+E_0))=2,$
	we get $h^0(\Omega_X(logD_2)\otimes \omega_X(L_2))=1$.
\end{proof}

\begin{lemma}
	\label{term 3}
	$h^0(\Omega_X(logD_3)\otimes \omega_X(L_3))=1$.
\end{lemma}
\begin{proof}  Since $D_3=0$ and $L_3\equiv 3D_0-E_{\eta_i}$, we have
	$$h^0(\Omega_X(logD_3)\otimes \omega_X(L_3))=h^0(\Omega_X(D_{\eta_i}))=h^0(\mathcal{O}_X(D_{\eta_i}))+h^0(\mathcal{O}_X(-D_0+E_{\eta_i}))=1.$$
\end{proof}

\begin{theorem}
	\label{irreducible component 1}
	We have   $h^1(T_S)= 5=\dim\mathcal{M}'_1$.  Therefore $\overline{\mathcal{M}'_1}$ is an irreducible component of $\mathcal{M}^{4,3}_{1,1}$.
\end{theorem}

\begin{proof}
	By  Lemmas \ref{term 0}-\ref{term 3}, we have
	$h^0(\Omega_S\otimes \omega_S)=h^0(\Omega_X(logD_1, logD_2, logD_3)+ \sum_{i=1}^3 h^0(\Omega_X(logD_i)\otimes \omega_X(L_i))=3$. By Reimann-Roch and Serre duality, we have  $h^1(T_S)=h^2(T_S)+2=h^0(\Omega_S\otimes \omega_S)+2=5$.  By Propositon \ref{dim M'1}, we have $h^1(T_S)= 5=\dim\mathcal{M}'_1$. Hence $\overline{\mathcal{M}'_1}$ is an irreducible component of $\mathcal{M}^{4,3}_{1,1}$.
\end{proof}

\section{Surfaces of type $I_2$}
In this section, we study  surfaces of type $I_2$. The method is similar to that for surfaces of type $I_1$.    We omit the proof wherever  it is similar to that for  surfaces of type $I_1$.

\subsection{Bidouble covers of $B^{(2)}$}
As before, let $B^{(2)}$ be the second symmetric product of an  elliptic curve $B$.  Let  $C_{\eta_i}:=\{(x,x+\eta_i), x\in B\}$ $(i=1,2,3)$  be the (only) three curves homologous to $-K_X$ (see \cite{Cat81} proposition 7).   Let $\tau$ be a  point on $B$ such that $2\tau\equiv \eta_1+\eta_2$.
\begin{theorem}
	\label{type 2}
	Let $h: S'\rightarrow B^{(2)}$ be a  bidouble cover determined by   effective   divisors $D_1\equiv2D_0, D_2\equiv 4D_0-E_{\eta_1}-E_{\eta_2},D_3=0$, and divisors $L_1\equiv 2D_0-E_{\tau}, L_2\equiv D_0, L_3\equiv 3D_0-E_{\tau}$  such that $S'$ has at most RDP's as singularities.  Then the minimal resolution  $S$ of $S'$  is a surface of type $I_2$.  Conversely, for any surface $S$ of type $I_2$, the canonical model  $S'$ of $S$  is a bidouble cover of  $B^{(2)}$  where $B=Alb(S)$) determined by the effective   divisors $(D_1,D_2,D_3)$ and divisors  $(L_1,L_2,L_3)$ above.
\end{theorem}
\begin{proof}
	The proof is similar to Theorem \ref{type 1}.  Note  here  $h^0(4D_0-E_{\eta_1}-E_{\eta_2})=H^0(S^4(E_{[0]}(2,1)(-\eta_1-\eta_2))=1$ (cf. Lemma \ref{order of torsion}).  And $C_{\eta_1}+C_{\eta_2}$ is the unique effective divisor in $|4D_0-E_{\eta_1}-E_{\eta_2}|$, which is the disjoint union of two smooth elliptic curves (see \cite{Cat81} proposition 7).
\end{proof}

\begin{remark}
	\label{dimension M2}
	By Theorem \ref{type 2} and using a similar calculation to Lemma \ref{dim M1}, we have
	$$\dim\mathcal{M}_{I_2}=1+\dim |D_1|+\dim |D_2|=1+2+0=3.$$
\end{remark}

\subsection{Natural deformations of smooth bidouble covers}
Let $S$ be a general surface of type $I_2$. Then we have a smooth bidouble cover $h: S\rightarrow X=B^{(2)}$  determined by branch  divisors $D_1\equiv2D_0,D_2\equiv 4D_0-E_{\eta_1}-E_{\eta_2},D_3=0$, and divisors $L_1\equiv 2D_0-E_{\tau}, L_2\equiv D_0, L_3\equiv 3D_0-E_{\tau}$. In this subsection, we study the natural deformations of $S$. The method is the same to that of section 4.2.

  Since $D_3=0$,   $h$ is a simple bidouble cover.
Let $L_1':=D_0, L_2':=2D_0-E_\tau$ and let $z_1,z_2$ be the fibre coordinates relative to the two summands of $V:=\oplus _{j=1}^2\mathcal{O}_X(-L_j')$. Let $x_i$ be a section of $\mathcal{O}_X(D_i)$ with div$(x_i)=D_i$ ($i=1,2$). By  \cite{Cat85} p. 75,  $S$ is a subvariety of $V$ defined by equations:

\begin{equation}
\label{equation  simple bidouble cover 2}
\quad z^2_1=x_1,
\quad  z_2^2=x_2.
\end{equation}
and a natural deformation $Y$ of $S$ is defined by equations
\begin{equation}
\label{equation natural deformation simple bidouble cover 2}
\quad z^2_1=x_1 +b_1z_2,
\quad  z_2^2=x_2+b_2z_1.
\end{equation}
with $b_1\in H^0(\mathcal{O}_X(D_1-L_2'))$, $b_2\in H^0(\mathcal{O}_X(D_2-L_1'))$.

Note that $h^0(\mathcal{O}_X(D_1-L_2'))=h^0(E_\tau)=h^0(\mathcal{O}_B(\tau))=1$.  By \cite{CC93} Theorem 1.13, we have $H^0(\mathcal{O}_X(D_2-L_1'))=H^0(3D_0-E_{\eta_1}-E_{\eta_2})=0$, hence we always have $b_2=0$.

\vspace{3ex}
Denote by $M'_2$ the family of all surfaces arising as natural deformations of some general surface of type $I_2$, and by $\mathcal{M}'_2$ the image of $M'_2$ in $\mathcal{M}_{1,1}^{4,3}$. Let $\overline{\mathcal{M}'_2}$ be the Zariski closure of $M'_2$ in $\mathcal{M}_{1,1}^{4,3}$. Then we have
\begin{pro}
	\label{dim M'2}
	$\dim \mathcal{M}'_2=4$ and  $\mathcal{M}_{I_1}$ is a 3-dimensional subspace of $\overline{\mathcal{M}'_2}$.
\end{pro}
\begin{proof}
	Since there is one  parameter for $X=B^{(2)}$ (see Lemma \ref{dim M1}). From equations (\ref{equation natural deformation simple bidouble cover 2}), we see that $\dim \mathcal{M}'_2=1+\dim |D_1|+ \dim |D_2|+ h^0(\mathcal{O}_X(D_1-L_2'))=1+2+0+1=4$.
\end{proof}

\begin{remark}
	\label{remark general surface in M'2}
	From equations (\ref{equation natural deformation simple bidouble cover 2}), It is easy to see that a natural deformation $Y$ of $S$ is a bidouble cover of $X$ if and only if $b_1=0$ (since we always have $b_2=0$). By Theorem \ref{type 2}, $Y$ has a genus 3 hyperelliptic Albanese fibration if and only if $b_1=0$.  Since $\dim \mathcal{M}_{I_2}<\dim \mathcal{M}'_2$, we see that  a general surface in $M'_2$ has a genus  3 nonhyperelliptic Albanese fibation.
\end{remark}

\subsection{$h^1(T_S)$ for a general surface $S$ of type $I_2$}
Let $S$ be a general surface of type $I_2$. In this subsection we calculate $h^1(T_S)$.
Note that a general surface $S$ of type $I_2$ is a smooth bidouble cover of $X=B^{(2)}$ determined by effective divisors $(D_1,D_2,D_3)$ and divisors $(L_1,L_2,L_3)$ as in Theorem \ref{type 2}.

By Riemann-Roch,  we have   $h^1(T_S)=h^2(T_S)+2=h^0(\Omega_S\otimes \omega_S)+2$.  By \cite{Cat84} Theorem 2.16,  we have
$$H^0(\Omega_S\otimes \omega_S)\cong H^0(h_*(\Omega_S\otimes \omega_S))$$
$$=H^0(\Omega_X(logD_1, logD_2, logD_3)\otimes\omega_X)\oplus(\bigoplus_{i=1}^3 H^0(\Omega_X(logD_i)\otimes \omega_X(L_i)).$$

\begin{lemma}
	\label{term 0-3}
	(1) $h^0(\Omega_X(logD_1, logD_2, logD_3)\otimes\omega_X)=0$;
	
	(2) $h^0(\Omega_X(logD_1)\otimes \omega_X(L_1))=0$;
	
	(3) $h^0(\Omega_X(logD_2)\otimes \omega_X(L_2))=1$;
	
	(4)	$h^0(\Omega_X(logD_3)\otimes \omega_X(L_3))=1$.
\end{lemma}
\begin{proof}
	(1) Consider the exact sequence
	$$0\rightarrow \Omega_X\otimes \omega_X\rightarrow \Omega_X(logD_1, logD_2, logD_3)\otimes\omega_X\rightarrow \mathcal{O}_{D_1}(K_X)\oplus \mathcal{O}_{T_{\eta_i}}(K_X) \oplus \mathcal{O}_{T_{\eta_j}}(K_X)\rightarrow 0.$$
	Since  $K_XD_1=-2$, we have  $h^0( \mathcal{O}_{D_1}(K_X))=0$.
	From the exact sequence
	$$0\rightarrow \mathcal{O}_X(E_0-E_{\eta_i})\rightarrow \mathcal{O}_X(K_X)\rightarrow \mathcal{O}_{T_{\eta_i}}(K_X)\rightarrow 0,$$
	$h^0(\mathcal{O}_X(E_0-E_{\eta_i}))=h^1(\mathcal{O}_X(E_0-E_{\eta_i}))=0$ and $h^0(\mathcal{O}_X(K_X))=0$, we see $h^0(\mathcal{O}_{T_{\eta_i}}(K_X))=0$. Similarly, we have  $h^0(\mathcal{O}_{T_{\eta_j}}(K_X))=0$ and hence
	$h^0(\mathcal{O}_{D_1}(K_X)\oplus \mathcal{O}_{T_{\eta_i}}(K_X) \oplus \mathcal{O}_{T_{\eta_j}}(K_X))=0$.
	Since moreover $h^0(\Omega_X\otimes \omega_X)=0$,  we get   $h^0(\Omega_X(logD_1, logD_2, logD_3)\otimes\omega_X)=0$.
	
	(2) The proof is similar to  Lemma \ref{term 1},  just replace $\eta_i$ by $\tau$.
	
	(3) Consider the smooth double cover  $g: Y\rightarrow X$ with $g_*\mathcal{O}_Y=\mathcal{O}_X\oplus \mathcal{O}_X(-\epsilon)$, where  $\epsilon\equiv 2D_0-E_{\tau}$.  Since $K_Y\equiv g^*(K_X+\epsilon)\equiv g^*(E_0-E_{\tau})$ and $h^0(K_Y)=h^0(K_X)+h^0(K_X+\epsilon)=0$, we have   $K_Y\not\equiv 0$ and  $4K_Y\equiv 0$. Moreover, we have  $q(Y)=h^1(\mathcal{O}_Y)=h^1(\mathcal{O}_X)+h^1(\mathcal{O}_X(-\epsilon))=1$.  Thus  $Y$ is a  bielliptic surface and its  Albanese map  $\alpha_Y: Y\rightarrow B$ is a smooth map.  Hence  $\Omega_{Y/B}$ is a locally free sheaf and  we have the following exact sequence:
	$$0\rightarrow \alpha_Y^*\omega_B \rightarrow \Omega_Y \rightarrow \Omega_{Y/B} \rightarrow 0.$$
	Since $\omega_B\cong \mathcal{O}_B$, we have $\alpha_Y^*\omega_B \cong \mathcal{O}_Y$ and thus  $\Omega_{Y/B}\cong \omega_Y$.  Now  the above exact sequence  becomes
	$$0\rightarrow \mathcal{O}_Y \rightarrow \Omega_Y \rightarrow \omega_Y \rightarrow 0.$$
	Tensoring this exact sequence  with $\omega_Y(g^*D_0)$, we get
	$$0\rightarrow \mathcal{O}_Y(g^*(D_0+E_0-E_{\tau})) \rightarrow \Omega_Y \otimes \omega_Y(g^*D_0) \rightarrow \mathcal{O}_Y(g^*(D_0+2E_0-2E_{\tau})) \rightarrow 0.$$
	Since $h^0(\mathcal{O}_Y(g^*(D_0+E_0-E_{\tau}))=h^0(\mathcal{O}_Y(g^*(D_0+2E_0-2E_{\tau})))=1$ and $h^1(\mathcal{O}_Y(g^*(D_0+E_0-E_{\tau}))=0$, we get  $h^0(\Omega_Y \otimes \omega_Y(g^*D_0))=2$.
	
	On the other hand, by \cite{Cat84} Proposition 3.1, we have  $H^0(\Omega_Y \otimes \omega_Y(g^*D_0))\cong H^0(\Omega_X(logD_2)\otimes \omega_X(D_0)) \oplus H^0(\Omega_{X}\otimes \omega_X(\epsilon+D_0))$.  Since $h^0(\Omega_X\otimes \omega_X(\epsilon+D_0))=h^0(\mathcal{O}_X(D_0+E_0-E_{\tau}))+h^0(\omega_X(D_0+E_0-E_\tau))=1$, we get $h^0(\Omega_X(logD_2)\otimes \omega_X(L_2))=h^0(\Omega_X(logD_2)\otimes \omega_X(D_0))=1$.

	(4) $h^0(\Omega_X(logD_3)\otimes \omega_X(L_3))=h^0(\mathcal{O}_X(D_0+E_0-E_{\tau}))+h^0(\omega_X(D_0+E_0-E_{\tau}))=1$.
\end{proof}

\begin{theorem}
	\label{irreducible 2}
	We have  $h^1(T_S)=4=\dim\mathcal{M}'_2$. Therefore $\overline{\mathcal{M}'_2}$ is an irreducible component of $\mathcal{M}^{4,3}_{1,1}$.
\end{theorem}
\begin{proof} By  Lemma \ref{term 0-3}, we have  $h^0(\Omega_S\otimes \omega_S)=h^0(\Omega_X(logD_1, logD_2, logD_3)\otimes\omega_X)+\sum_{i=1}^3 h^0(\Omega_X(logD_i)\otimes \omega_X(L_i))=2$. By Riemann-Roch, Serre duality and Proposition \ref{dim M'2}, we have   $h^1(T_S)=h^2(T_S)+2=h^0(\Omega_S\otimes \omega_S)+2=4=\dim \mathcal{M}'_2$. Hence $\overline{\mathcal{M}'_2}$ is an irreducible component of $\mathcal{M}^{4,3}_{1,1}$.
\end{proof}

\vspace{3ex}
$\mathbf{Acknowledgements.}$
The author was sponsored by China Scholarship Council ``High-level university graduate program''(No.201506010011).

The author would like to thank his advisor, Professor Fabrizio Catanese    at Universit\"at Bayreuth for suggesting this research topic,  for a lot of  inspiring discussion with the author  and for his encouragement  to the author.  The author would also like to thank his domestic advisor, Professor Jinxing Cai  at Peking University for his encouragement and some useful suggestions.
The author is grateful to Binru Li  for a lot of helpful discussion. Thanks also goes to Stephen Coughlan and Andreas Demleitner   for improving the English.


\vspace{3ex}
School of Mathematical Sciences, Peking  University, Yiheyuan Road 5, Haidian District, Beijing 100871,  People's Republic of China

E-mail address: 1201110022@pku.edu.cn

\vspace{2ex}
Lehrstuhl Mathematik VIII, Universit\"at Bayreuth,  NW II,  Universit\"atsstr. 30, 95447 Bayreuth, Germany

E-mail address: songbo.ling@uni-bayreuth.de

\end{document}